\documentclass[11pt]{amsart}

\usepackage{amssymb, graphicx}
\usepackage{titletoc} 
\usepackage{bbm} % for the complex imaginary unit i
\usepackage[normalem]{ulem}

\renewcommand{\subset}{\subseteq}

\usepackage{enumerate}
\usepackage[usenames,dvipsnames]{color}

\makeatletter
\def\moverlay{\mathpalette\mov@rlay}
\def\mov@rlay#1#2{\leavevmode\vtop{%
   \baselineskip\z@skip \lineskiplimit-\maxdimen
   \ialign{\hfil$#1##$\hfil\cr#2\crcr}}}

\makeatother

\numberwithin{equation}{section}

\definecolor{Ruta}{rgb}{0.309, 0.459, 0.208}
\definecolor{Vino}{rgb}{0.256,0,0}
\definecolor{Siva}{rgb}{0.116,0.116,0.116}
\definecolor{Siva'}{rgb}{0.250,0.116,0.116}

\usepackage{amsfonts}
\usepackage{amsmath}
\usepackage{verbatim}

\usepackage[pagebackref,colorlinks,linkcolor=red,citecolor=blue,urlcolor=blue,hypertexnames=true]{hyperref}

\usepackage{blkarray}
\usepackage{multirow}
\usepackage{xspace} 
\usepackage{multirow,bigdelim}

\newcommand{\cB}{\mathcal{B}}

\newcommand{\cC}{\mathcal{C}}

\def\beq{ \begin{equation} }
\def\eeq{ \end{equation} }

\def\bep{\begin{proof}}
\def\eep{\end{proof}}

\def\ben{ \begin{enumerate} }
\def\een{ \end{enumerate} }

\newcommand{\id}{{\rm id}}

\newcommand{\tnz}{\otimes}
\newcommand{\ol}{\overline}
\newcommand{\tr}{\mathrm{tr}}

\allowdisplaybreaks[1]

\newtheorem{theorem}{Theorem}[section]
\newtheorem{proposition}[theorem]{Proposition}

\newtheorem{corollary}[theorem]{Corollary}

\theoremstyle{definition}

\newtheorem{lemma}[theorem]{Lemma}

\newtheorem{remark}[theorem]{Remark}

\newtheorem{example}[theorem]{Example}

\newcommand{\X}{\langle x\rangle}
\newcommand{\Xt}{\langle x,x^t\rangle}

\newcommand{\cyc}{\stackrel{\mathrm{cyc}}{\thicksim}}

\newcommand{\s}{\sigma} 
\newcommand{\CC}{\mathbb{C}}
\newcommand{\RR}{\mathbb{R}}

\newcommand{\N}{\mathbb{N}}

\newcommand{\F}{\mathbb{F}}

\DeclareMathOperator{\GM}{GM}
\newcommand{\GMt}{\GM_n^\dagger}

\newcommand{\td}{\tilde}
\newcommand{\wtd}{\widetilde}

\newcommand{\cU}{{\mathcal U}}
\newcommand{\cV}{{\mathcal V}}
\newcommand{\cW}{{\mathcal W}}
\newcommand{\OO}{\mathrm O}
\newcommand{\GL}{\mathrm {GL}}
\newcommand{\U}{\mathrm {U}}
\newcommand{\D}{\mathrm D}
\newcommand{\Mr}{\cM(\RR)}
\newcommand{\MF}{\cM(\F)}

\newcommand{\Rn}{\mathrm{R}_n} 
\newcommand{\Rnt}{\Rn^\dagger}

\newcommand{\Rp}{\text{generalized polynomial}\xspace}  
\newcommand{\Rps}{\text{generalized polynomials}\xspace} 
\newcommand{\RPS}{\text{Generalized Polynomials}\xspace} 
\newcommand{\w}{\mathfrak{W}} 
\newcommand{\nA}{M_n(\F)\ast\F\X} 
\newcommand{\nAt}{M_n(\F)\ast\F\Xt} 
\newcommand{\GMn}{\mathrm{g}\!\GM_{ns}} 
\newcommand{\Rns}{\mathrm{g}\mathrm{R}_{ns}}  
\newcommand{\Ftr}{T\X}  

\newcommand{\cT}{{\mathcal T}}
\newcommand{\cgl}{\text{$(C_{\GL_n}(B),\GL)$-concomitant}\xspace}
\newcommand{\cg}{\text{$(C_{G_n}(B),G)$-concomitant}\xspace}
\newcommand{\co}{\text{$(C_{\OO_n}(B),\OO)$-concomitant}\xspace}

\usepackage{graphicx}
\usepackage{calc}
\newlength{\depthofsumsign}
\newcommand{\nsum}[1][1.4]{% only for \displaystyle
    \mathop{%
        \raisebox
            {-#1\depthofsumsign+1\depthofsumsign}
            {\scalebox
                {#1}
                {$\displaystyle\sum$}%
            }
    }
}

\makeatletter
\newcommand*{\rom}[1]{\expandafter\@slowromancap\romannumeral #1@}
\makeatother

\title[Free function theory through matrix invariants]{Free function theory through matrix invariants}
\author[I. Klep]{Igor Klep${}^{1}$}
\address{Igor Klep, Department of Mathematics, 
The University of Auckland, New Zealand}
\email{igor.klep@auckland.ac.nz}
\thanks{${}^1$Supported by the Marsden Fund Council of the Royal Society of New Zealand. Partially supported by the Faculty Research Development Fund (FRDF) of The
University of Auckland (project no. 3701119). Partially supported by the Slovenian Research Agency grants P1-0222, L1-4292 and L1-6722. Part of this research was done while the author was on leave from the University of Maribor.}

\author[\v S. \v Spenko]{\v Spela \v Spenko${}^{2}$}
\address{\v Spela \v Spenko,  Institute of  Mathematics, Physics, and Mechanics,  Ljubljana, Slovenia} \email{spela.spenko@imfm.si}
\thanks{${}^{2}$Supported by the Slovenian Research Agency and in part by the Slovene Human Resources Development and Scholarship Fund.}

\date{\today}

\begin{document}

\setcounter{tocdepth}{3}
\contentsmargin{2.55em} 
\dottedcontents{section}[3.8em]{}{2.3em}{.4pc} 
\dottedcontents{subsection}[6.1em]{}{3.2em}{.4pc}
\dottedcontents{subsubsection}[8.4em]{}{4.1em}{.4pc}

\makeatletter
\newcommand{\mycontentsbox}{%
{\centerline{NOT FOR PUBLICATION}
\addtolength{\parskip}{+.3pt}
\small\tableofcontents}}
\def\enddoc@text{\ifx\@empty\@translators \else\@settranslators\fi
\ifx\@empty\addresses \else\@setaddresses\fi
\newpage\mycontentsbox\newpage\printindex}
\makeatother

\begin{abstract}
%This paper concerns free function theory and various types of free maps.
%Examples of these maps include noncommutative polynomials, rational functions, and convergent power series. %\looseness=-1
%
In this article we introduce powerful tools and techniques from invariant theory to free analysis. 
This enables us to study free maps with involution. 
These maps are free noncommutative analogs of real analytic functions of several variables. 
With  examples we demonstrate that 
they do not exhibit strong rigidity properties of their involution-free free counterparts. 
We present a characterization of polynomial free maps via  properties of their finite-dimensional slices. 
This is used to establish power series expansions for analytic free maps about scalar and non-scalar points; the latter are given by series of generalized polynomials for which we obtain an invariant-theoretic characterization.
  Finally, we  give an inverse and implicit  function theorem for free maps with involution. 
\end{abstract}

\maketitle

\section{Introduction}

The notion of a free map arises naturally in free probability, the study of noncommutative rational functions \cite{AlpDu,BGM,HMV}, and systems theory \cite{HBJP, KV0}. 
The study of these maps is in the realm of free analysis \cite{AM1,AM2,AY, AK,BV,KV,HKM12,MS,Pas,PT,Po1,Po2,Tay,Voc04,Voc10}. 

The main contribution of this paper is to introduce powerful invariant-theoretic methods \cite{Pro} to free analysis. 
We present an alternative, algebraic approach to free function theory. 
While most of the current efforts in free analysis are focused on (involution-free) free maps -- free analogs of analytic functions in several complex variables -- where strong rigidity is observed, our main attention is to {\em free maps with involution}, e.g.~ noncommutative polynomials, rational function or power series in $x,x^t$. 
Our methods are uniform in that they work in both cases with only minimal adaptations needed. 
Thus we recover some of the existing results on (involution-free) free maps (cf. \cite{AM2,KV,Pas}).

 We next give a list of the main results, that at the same time serves as a guide to the paper; we refer to Section \ref{sec2} for definitions and unexplained terminology. 
\begin{enumerate}[\rm (1)]
\item
A free map with involution $f$ is a polynomial in $x,x^t$ if and only if there is $d\in \N$ such that each of the level functions $f[n]$ is a polynomial of degree $\leq d$ (Proposition \ref{hompol});
\item
Analytic free maps with involution admit convergent power series expansions 
about scalar points (Theorem \ref{analit});
\item
Analytic free maps with involution admit convergent power series expansions about non-scalar points (Theorem \ref{rumanalit}, Theorem \ref{rumanalitO}), whose  homogeneous parts are generalized polynomials. We present  an invariant theoretic characterization of the latter in Subsection \ref{ssec4};
\item
Free inverse and implicit function theorems for differentiable free maps with involution are the theme of Section \ref{sec5}, see Theorem \ref{IFFT}, Corollary \ref{IFFT2}, and Theorem \ref{IFFTn};
\item
Section \ref{ex} presents several illustrating examples demonstrating non-rigidity properties 
of free maps with involution. For instance, we give an example of a bounded smooth free map with involution that is not analytic (Example \ref{sin}).
\end{enumerate}

\def\cM{\mathcal M}

\section{Preliminaries}\label{sec2}
In this section we present preliminaries from free analysis, polynomial identities \cite{Dre,Row} and invariant theory \cite{Pro}
needed in the sequel.

\subsection{Notation}
Let $\F\in \{\RR,\CC\}$ and let $\cM(\F)$ stand for $\bigcup_n  M_n(\F)$. We denote the  monoid generated by $x_1,\dots,x_g$ by $\X$, and the free associative algebra in the variables $x=(x_1,\dots,x_g)$ by $\F\X$. 
The free algebra with involution in the variables $x_1,x_1^t,\dots,x_g,x_g^t$ is denoted by $\F\Xt$.
The elements of degree $d$ in $\F\X$ (resp. $\F\Xt$) are denoted by $\F\X_d$ (resp. $\F\Xt_d$). 
We write 
\[C=\F\big[x_{ij}^{(k)}\mid 1\leq i,j\leq n,1\leq k\leq g\big]  
\]
for the commutative polynomial ring in $gn^2$ variables.
We  equip $M_n(C)$ with the transpose involution fixing $C$ pointwise. 
The matrices $X_k=(x_{ij}^{(k)})\in M_n(C)$, $1\leq k\leq g$, are called {\bf generic matrices}. 
By $\GM_n$ 
we denote the {unital} subalgebra of $M_n(C)$ generated by generic matrices, 
and by  $\GMt$  the subalgebra of $M_n(C)$ generated by  generic matrices and their transposes.
We let $\Rn$ stand for the subalgebra of $M_n(C)$ generated by the generic matrices and traces $\tr(X_{i_1}\cdots X_{i_k})$ {of their products},  
and  $\Rnt$ for the subalgebra of $M_n(C)$ generated by  generic matrices,  their transposes, and traces $\tr(U_{i_1}\cdots U_{i_k})$, $U_j\in \{X_j,X_j^t\}$.
The center of $\Rn$ (resp. $\Rnt$) is generated by the traces, we denote it by $Z(\Rn)$ (resp. $Z(\Rnt)$).

\subsection{Free Sets and Free Maps}
Let $G=(G_n)_n$ be a sequence of groups {with $G_n\subseteq\GL_n(\F)$}, 
satisfying \beq\label{grgp}
G_n \oplus G_m = \begin{pmatrix} G_n & 0 \\ 0 &G_m \end{pmatrix} \subseteq G_{n+m}.
\eeq
 We will be primarily concerned with the case $G_n=\GL_n(\F)$ for all $n$, or
$G_n$ is the orthogonal group $\OO_n(\RR)$ for all $n$. The modifications needed for the case
of the unitary groups $G_n=\U_n(\CC)$ will be discussed in  Appendix \ref{apU}.  
For simplicity of notation we write $\GL_n,\OO_n,\U_n$ instead of $\GL_n(\F),\OO_n(\RR),\U_n(\CC)$, respectively. 
Let us denote $\GL=(\GL_n)_{n\in \N}$, $\OO=(\OO_n)_{n\in\N}$, $\U=(\U_n)_{n\in \N}$. 
A subset $\cU\subseteq \cM(\F)^g$ is a sequence $\cU=(\cU[n])_{n\in \N}$, where each $\cU[n]\subseteq M_n(\F)^g$. The set $\cU$ is a {\bf $G$-free set}
if it is closed with respect to  simultaneous $G$-similarity  and with respect to direct sums; i.e., for every $m,n\in \N$:
\begin{equation}\label{eq:conj}
\s X\s^{-1}=(\s X_1\s^{-1},\dots,\s X_g\s^{-1})\in \cU[n]
\end{equation}
for all $X\in \cU[n]$, $\s\in G_n$, and 
\begin{equation}\label{eq:oplus}
X\oplus Y=
\begin{pmatrix}
X&0\\
0&Y
\end{pmatrix}\in \cU[m+n]
\end{equation}
for all $X\in \cU[m], Y\in \cU[n]$. 

Let $\cU$ be a $G$-free set. 
We call a sequence of functions $f=(f[n])_{n\in\N}:(\cU[n])_{n\in\N}\to \MF$ a {\bf $G$-free map}, if it respects $G$-similarity and 
direct sums; i.e, for every $m,n\in \N$:
\begin{equation}\label{eq:conj2}
f[n](\s X\s^{-1})=\s\, f[n](X)\,\s^{-1} 
\end{equation}
for all $X\in \cU[n]$, $\s\in G_n$, and
\begin{equation}\label{eq:oplus2}
f[m+n](X\oplus Y)=f[m](X)\oplus f[n](Y) 
\end{equation}
for all $X\in \cU[m], Y\in \cU[n]$.
{In the language of invariant theory} \cite{Pro,KP} the condition \eqref{eq:conj2} says that $f[n]$ is a $G_n$-concomitant. If $f$ satisfies  only \eqref{eq:conj2} for all $n$ (and not necessarily \eqref{eq:oplus2}) we call it a {\bf free $G$-concomitant}. 
Sometimes a $\GL$-free map is called simply a {\bf free map} and an $\OO$-free map is a {\bf free map with involution}.

With a slight abuse of notation we sometimes also refer to a map $f:\cU\to \cM$ as a $G$-free map  if its domain $\cU$ is only closed under direct sums, $f$ respects direct sums  and $f$ respects $G$-similarity on $\cU$; i.e, for every  $n\in \N$:
$$f[n](\sigma X\sigma^{-1})=\sigma\, f[n](X)\,\sigma^{-1}$$
 for all $X\in \cU[n]$, $\sigma\in G_n$ such that $\sigma X\sigma^{-1}\in \cU[n]$. 
In this case we can canonically  extend $f$ to the similarity invariant envelope of $\cU$
 (cf.~\cite[Appendix A]{KV}), and remain in the framework of the given definition:

\begin{proposition}\label{simenv}
Let $\cU\subseteq \cM(\F)^g$ be closed under direct sums, and let $f:\cU\to \cM(\F)$ respect direct sums and $G$-similarity on $\cU$. Then 
$$\td\cU=\{\sigma A\sigma^{-1}\mid A\in \cU[n], \sigma\in G_n,n\in \N\}$$
is a $G$-free set, and there exists a unique $G$-free map $\td f:\td\cU\to\cM(\F)$ such that $\td f|_\cU=f$, defined by  $\td f(\s X\s^{-1})=\s f(X)\s^{-1}$ for $X\in \cU[n]$, $\s\in G_n$. 
\end{proposition}

\begin{remark}
In \cite[Appendix A]{KV} the proof is given in the case $G=\GL$. The same proof with obvious modifications works also for any sequence of groups $G=(G_n)_n$ satisfying \eqref{grgp}, 
in particular for $G\in\{\OO,\U\}$.
\end{remark}

A $G$-free map $f$ is $\F$-analytic  around $0$ if there exists a neighborhood 
\beq\label{eq:nbhd}
\cB(0,\delta)=\bigcup_n\{X\in M_n(\F)^{g}\mid \|X\|<\delta_n\}
\eeq
 of $0$ in $\cM(\F)^g$ such that $f[n]_{ij}$ is $\F$-analytic on $\cB(0,\delta)[n]$, $\delta=(\delta_n)_{n}$, and $\delta_n>0$ for every $n\in \N$.
It is a polynomial map of degree $m$ if $f[n]_{ij}$ are  polynomials in $x_{ij}^{(k)}$ of degree $\leq m$ and at least one of the polynomials $f[n]_{ij}$ is of degree $m$; it is homogeneous of degree $m$  if $f[n]_{ij}$ are homogeneous polynomials of degree $m$ or zero polynomials, and $f[n]_{ij}$ is of degree $m$ for  at least one triple $(n,i,j)$.

\subsection{Trace Polynomials}
The free algebra  with trace $\Ftr$  is the algebra of free noncommutative polynomials in the variables $x_k$  over the polynomial algebra $T$  in the infinitely many variables  $\tr(w)$, where $w$ runs over all representatives of the cyclic equivalence classes of words in the variables $x_k$; i.e., $w\in \X/_{\cyc}$. 
Here two words $u,v\in\X$ are cyclically equivalent, $u\cyc v$, iff $u$ is a cyclic permutation of $v$.
The free $*$-algebra with trace $T^\dagger\Xt$  is the algebra of free noncommutative polynomials in the variables  $x_k, x_k^t$
over the polynomial algebra  $T^\dagger$  in the infinitely many variables  $\tr(w)$, where $w$ runs over all representatives of the {$*$-cyclic} equivalence classes of words in the variables $x_k,x_k^t$; i.e., words $u$ and $v$ are equivalent if $u\cyc v$ {or $u^t\cyc v$}. 
The elements of $\Ftr$ (resp. $T^\dagger\Xt$) are {\bf trace polynomials} (resp. {\bf trace polynomials with involution}) and elements of $T$ (resp. $T^\dagger$) are {\bf pure trace polynomials} (resp. {\bf pure trace polynomials with involution}).
The degree of a trace monomial $\tr(w_1)\cdots\tr(w_m)v$, $w_i,v\in \X$,  equals $|v|+\sum_i |w_i|$, where $|u|$ denotes the length of a word $u$.
The degree of a trace polynomial is the maximum of the degrees of its trace monomials.

{\em Trace identities} of the matrix algebra $M_n(\F)$ (with involution) are the elements in the kernel of the evaluation map from the free algebra (with involution) with trace to $M_n(\F)$; i.e., trace identities of $M_n(\F)$ are trace polynomials that vanish on $n\times n$-matrices. {\em Pure trace identities} are trace identities that belong to $T$ (resp. $T^\dagger$).

The free ($*$-)algebra with trace $\Ftr$ (resp. $T^\dagger \Xt$) and the trace identities have its interpretation in terms of invariants of matrices. Let $G=\GL_n$ (resp. $G=\OO_n$)  act by conjugation on $M_n(\F)$
and diagonally (i.e., componentwise) on $M_n(\F)^g$.
The first fundamental theorem for matrices (with involution)  yields that a $\GL_n$- (resp. $\OO_n$-) concomitant is a trace polynomial (resp. with involution), see \cite[Theorem 2.1, Theorem 7.2]{Pro} or \cite[Chapter 11]{Pro3} for a broader perspective on the subject. 
(For another take on the theory of polynomial identities we refer the reader to \cite{BCM}.) 
Viewing a polynomial map $f:M_n(\F)^g\to M_n(\F)$ as an element $\td f\in M_n(C)$ we can see that the algebra of $\GL_n$- (resp. $\OO_n$-) concomitants is isomorphic to $\Rn$ (resp. $\Rnt$), and $\Rn$ (resp. $\Rnt$) is isomorphic to the quotient of $\Ftr$ (resp. $T^\dagger \Xt$) by the ideal of trace identities (resp. trace identities  with involution).

\section{Analytic $G$-Free Maps and Power Series Expansions about Scalar Points}\label{0}

{In this section we investigate two distinguished classes of free maps,
namely polynomials and analytic free maps. We characterize free maps which
are polynomials in Subsection \ref{subsec:poly}, and use this to show that analytic free maps
admit power series expansions about scalar points 
in Subsection \ref{subsec:analytic}.}
{These results are classical for $G=\GL$ (cf.~\cite{KV,Tay,Voc10}) and are -- to the best of our knowledge -- new for $G=\OO$.}
Throughout this section $G\in\{\GL,\OO\}$.

\subsection{Polynomial Free Maps}\label{subsec:poly}

{We start by characterizing free polynomial maps $f$ via their ``slices'' $f[n]$.
For $G=\GL$ this result is due to 
Kaliuzhnyi-Verbovetskyi and Vinnikov \cite[Theorem 6.1]{KV} who deduce it from their power series expansion theorem for analytic free maps. 
In contrast to this we shall first characterize free polynomial maps and employ this in Subsection \ref{subsec:analytic} to establish power series expansions for analytic $G$-free maps. 
 Our proofs are uniform in that they work for both $G=\GL$ and $G=\OO$, and are purely algebraic,
depending only on the invariant theory of  matrices   \cite{Pro}.}

\begin{comment}
\begin{proposition}\label{prop:freetr}
Let $f:\MF^g\to \MF $ be a free $G$-concomitant. 
If $f$ is a polynomial map
and $\max_n \deg f[n]= d$, then $f$ is a trace polynomial of degree $d$.
\end{proposition}
\end{comment}
%\begin{proof}

%\end{proof}

\begin{proposition}\label{hompol}
Let $f:\MF^g\to \MF $ be a $G$-free map. 
If $f$ is a polynomial map
and $\max_n \deg f[n]= d$, then $f$ is a free 
polynomial of degree $d$.
{That is, $f\in\F\X_d$ if $G=\GL$ and
$f\in\F\Xt_d$ if $G=\OO$.}
\end{proposition}

\begin{proof}
Since  $f[n]:M_n(\F)^g\to M_n(\F)$ is a concomitant, it follows by \cite[Theorem 2.1, Theorem 7.2]{Pro} that $f[n]$ is a trace polynomial of  degree $\leq d$ in the variables $x_k$ (resp. $x_k, x_k^t$). 
Since there do not exist nontrivial trace identities for $M_n(\F)$ of degree less than $n$ by \cite[Theorem 4.5, Proposition 8.3]{Pro} (see also \cite{BK,Raz}), we can write $f[n]$ in the case $n\geq d+1$ uniquely as 
$$f[n]=\sum_M \tr(h_M^n)M,$$ 
where $M$ runs over all monomials of degree $\leq d$ and $\deg \tr(h_M^n)+\deg M\leq d$. Choose $n\geq d+1$.
As $f$ is a free map, we have 
\begin{multline*}
\sum_M \tr\big(h_M^{2n}(X\oplus Y)\big)M(X)\oplus \sum_M \tr\big(h_M^{2n}(X\oplus Y)\big)M(Y)=f[2n](X\oplus Y) \\
=f[n](X)\oplus f[n](Y)=\sum_M \tr\big(h_M^{n}(X)\big)M(X)\oplus \sum_M \tr\big(h_M^{n}(Y)\big)M(Y)
.
\end{multline*}
Comparing both sides of the above expression we obtain 
$$\tr\big(h_M^{2n}(X\oplus Y)\big)=\tr\big(h_M^{n}(X)\big)=\tr\big(h_M^{n}(Y)\big)$$
since $M_n(\F)$ does not satisfy a nontrivial trace identity of degree $d$.
Thus, 
$$\tr\big(h_M^{n}(X)\big)=\alpha=\tr\big(h_M^{n}(Y)\big)$$ 
for some $\alpha\in \F$. 
Hence, for every $n>N$, $f[n]\in \GM_n$ (resp. $f[n]\in \GMt$)
is represented by an element $\tilde f\in \F\langle X\rangle$ (resp. $\tilde f\in \F\Xt$) of degree $d$. 
Since $f$ is a free map, we can identify it with a free polynomial in the variables $x_k$ (resp. $x_k,x_k^t$).
\end{proof}

\begin{remark}
We note that Proposition \ref{hompol} holds also if $f$ is only defined on $\cB(0,\delta)$ (cf.~Proposition \ref{simenv}), since polynomial functions that agree on an open subset of $M_n(\F)^g$ represent the same function on $M_n(\F)^g$.
\end{remark}

\subsection{Analytic Free Maps}\label{subsec:analytic}

{We next turn our attention to analytic $G$-free maps. We show they admit unique convergent power series expansions about scalar points $a\in\F^g$,  extending classical results for $G=\GL$, cf.~\cite{Tay,Voc04,Voc10,KV,HKM12}.} {By a translation we may assume without loss of generality that $a=0$.}

\begin{theorem}\label{analit}
Let $\cU$ be a $G$-free set and $f:\cU\to \MF$ an $\F$-analytic $G$-free map, and
let  $\cB(0,\delta)\subseteq \cU$, where $\delta=(\delta_n)_{n\in\N}$, $\delta_n>0$ for every $n\in \N$. 
{Then} there exists a unique formal power series
\beq\label{eq:pw}
F=\sum_{m=0}^\infty \sum _{|w|=m}F_w w,
\eeq
where $w\in \X$ $($resp. $w\in \Xt)$, which converges in norm on $\cB(0,\delta)$, {with} $f(X)=F(X)$ for $X\in \cB(0,\delta)$.  
\end{theorem}

\begin{remark}
If $f$ is uniformly bounded, and $G=\GL$ then the convergence of the power series $F$ in \eqref{eq:pw} is uniform, cf.~\cite[Proposition 2.24]{HKM12}, while this conclusion does not hold when $G=\OO$. We present examples in Section \ref{ex}.
\end{remark}

We first prove the existence, the uniqueness will follow from Proposition \ref{unique} below.

\begin{proof}[Proof of the existence]
Since $f$ is analytic, there exists for every $X\in M_n(\F)^g$ a neighbourhood of $0$ such that the function $t\mapsto f[n](tX)$ is defined and analytic in that neighbourhood. 
Hence, $f[n](tX)$ can be expressed in that neighbourhood as a convergent power series of the form $\sum_{m=0}^\infty t^m f[n]_{m}(X)$, 
where $f[n]_{m}(X)$ is a function of $X$. 
Note that for $X\in \cB(0,\delta)$, this power series converges for $t=1$. 
{The} function $f[n]_{m}$ is a homogeneous polynomial function of degree $m$. 
Indeed, 
let $s\in \F$, $X\in M_n(\F)^g$ and choose $\delta'$ such that $tsX\in \cB(0,\delta)$ for $|t|\leq \delta'$. 
Then 
$$\sum_{m=0}^\infty t^m f[n]_{m}(sX)=f(tsX)=\sum_{m=0}^\infty (ts)^mf[n]_{m}(X),$$ 
and thus $f[n]_{m}(sX)=s^m f[n]_{m}(X)$.

Let us show that $f_{m}$ defined by $f_{m}[n]:=f[n]_{m}$ is an analytic free map. 
Choose $\delta'$ such that $tX,tY,\sigma tX\s^{-1}\in \cB(0,\delta)$ for  $|t|<\delta'$.
As $f$ is a free map we have 
\begin{multline*}
\sum_{m=0}^\infty t^m f[n+n']_{m}(X\oplus Y)=f[n+n'](tX\oplus tY)
\\
=f[n](tX)\oplus f[n'](tY)=\sum_{m=0}^\infty t^m \big(f[n]_{m}(X)\oplus f[n']_{m}(Y)\big),
\end{multline*}
and
$$\sum_{m=0}^\infty t^m \s\, f[n]_{m}(X)\, \s^{-1}= \s\, f[n](tX)\, \s^{-1}=f[n](t\s X\s^{-1})=\sum_{m=0}^\infty t^m f[n]_{m}(\s X\s^{-1})$$
for all $|t|<\delta'$, 
which implies that $f_{m}$ is a $G$-free map. 
By construction, $f_{m}$ is a homogeneous polynomial function of degree $m$ (or $0$) for every $m$. 
By Proposition \ref{hompol},  $f_m$ can be represented by a 
free polynomial in the variables $x_k$ (resp. $x_k,x_k^t$) of degree $m$. Thus, $f$ can be expressed as a power series in noncommuting variables, $F=\sum f_{m}$.
By construction, this power series converges on $\cB(0,\delta)$. 
\end{proof}

While the theories of $\GL$- and $\OO$-free maps enjoy certain similarities, 
there are also major differences. 
For instance, for $\GL$-free maps continuity implies analyticity  and there is a very useful formula \cite[Proposition 2.5]{HKM11}, \cite[Theorem 7.2]{KV}
connecting function values with the derivative:
\begin{equation}\label{der}
f
\begin{pmatrix}
X&H\\
0& X
\end{pmatrix}=
\begin{pmatrix}
f(X)&{\delta}f(X)(H)\\
0& f(X)
\end{pmatrix},
\end{equation}
where ${\rm \delta}f(X)(H)$ denotes the G\^{a}teaux 
{(directional)}
derivative of $f$ at $X$ in the direction $H$; i.e., 
\[{\rm \delta}f(X)(H)=\lim_{t\to 0}\frac{f(X+tH)-f(X)}{t}.\]
For $\OO$-free maps continuity does not imply differentiability; see Section \ref{ex} for examples. 
However, for differentiable 
$\OO$-free maps we do have an analog of  formula \eqref{der}, which can be deduced from \cite[Lemma 2.3, Proposition 2.5]{PT}, but we prove it {here} for the sake of completeness. We write $\D f$ for a derivative of $f$, it can be either the G\^ateaux or the Fr\'echet derivative. The Lie bracket $[a,B]$ stands for $([a_,B_1],\dots,[a,B_g])$, where $a\in M_n(\F)$, $B=(B_1,\dots,B_g)\in M_n(\F)^g$.  
\begin{lemma}
Let $f:\cU\to \cM(\F)$ be a real differentiable $G$-free map. 
Then the identity 
\begin{equation}
\D f(X)([a,X])=[a,f(X)]
\end{equation}
holds for all $X\in \cU[n]$, $a^t=-a\in M_n(\RR)$.
In particular, \begin{equation}\label{did}
\D f
\begin{pmatrix}
X_1&0\\
0& X_2
\end{pmatrix}
\begin{pmatrix}
0&X_1-X_2\\
X_1-X_2& 0
\end{pmatrix}=
\begin{pmatrix}
0&f(X_1)-f(X_2)\\
f(X_1)-f(X_2)& 0
\end{pmatrix}.
\end{equation}
\end{lemma}

\begin{proof}
Note that $e^{sa}$ is orthogonal for $a^t=-a\in M_n(\RR)$ and $s\in \RR$. 
Thus we have
$$
f(e^{-sa}Xe^{sa})=e^{-sa}f(X)e^{sa}
$$
for every $X\in \cU[n]$. 
Differentating with respect to $s$ at $0$ yields 
$$
\D f(X)([a,X])=[a,f(X)].
$$
Take 
$$a=\begin{pmatrix}
0&I_n\\
-I_n& 0
\end{pmatrix}\in M_{2n}(\RR),$$
where $I_n$ denotes the identity in $M_n(\RR)$. 
Setting 
$X=\begin{pmatrix}
X_1&0\\
0& X_2
\end{pmatrix}$ 
we get the identity \eqref{did}.
\end{proof}

We now show that the power series expansion is unique for a $G$-free function and  give a way to recover its coefficients.

\begin{lemma}\label{matenote}
If $f(X)=\sum _{|w|\leq m}F_w w$, where {the sum is over words}  in the variables $x_k$ (resp. $x_k,x_k^t$),  then we can obtain {the coefficients} $F_w$ by  evaluations of $f$ on $M_{m+1}(\F)$.
\end{lemma}

\begin{proof}
We proceed inductively. 
Assume that we can obtain coefficients of $f(X)=\sum _{|w|\leq k}F_w w$ for $k<m$ by evaluations of $f$ on $M_{k+1}(\F)$. 
The case $k=1$ is trivial. Suppose that $k=m$. 
Let us determine the coefficient at 
 $w=u_{i_1}^{j_1}\cdots u_{i_s}^{j_s}$, where $\sum_{k=1}^s j_k=m$ and $u_{i_k}\in \{x_{i_k},x_{i_k}^t\}$. 
We denote $s_k=\sum_{i=1}^{k} j_i$. 
Setting $a_i=0$ at the beginning, we define a $g$-tuple $(a_i)\in M_{m+1}(\F)^g$ as follows.  
We let $k$ {run}  from $1$ to $s$, and at step $k$ we replace $a_{i_k}$ by
\[
a_{i_k}=
 \begin{cases}
a_{i_k}+\sum_{u=s_{k-1}+1}^ {s_k}e_{u,u+1}&
\textrm{if $u_{i_k}=x_{i_k}$},\\[.1cm]
a_{i_k}+\sum_{u=s_{k-1}+1}^ {s_k}e_{u+1,u}&
\textrm{if $u_{i_k}=x_{i_k}^t$}.
\end{cases}
\]

We shall show that 
$\tr(f(a_1,\dots,a_g)e_{m+1,1})=F_w$.
We need to find the coefficient of $f(a_1,\dots,a_g)$ expressed in {the} standard basis $e_{ij}$, $1\leq i,j\leq m+1$, of $M_{m+1}(\F)$ at $e_{1,m+1}$. 
According to the definition of {the} $a_{i}$'s it suffices to show that $e_{1,m+1}$ can be obtained in only one way as a product of $\leq m$ matrix units from the set $S=\{e_{i,i+1},e_{i+1,i}\mid 1\leq i\leq m\}$. Note that the multiplication on the right of any matrix unit $e_{ij}$ by any  element of $S$ either 
 {increases or decreases} $j$ by $1$. In order to obtain $e_{1,m+1}$ as a product of $\leq m$ elements from $S$, we can thus only choose matrix units which {increase} the second subscript of the preceding matrix unit in the product. Hence, $e_{1,m+1}=e_{12}\cdots e_{m,m+1}$, and any other product of $\leq m$ elements from $S$ {will be different from} $e_{1,m+1}$. As each $e_{i,i+1}$ appears only in one of the $a_i,a_i^t$, $1\leq i\leq g$, the order $e_{12},\dots,e_{m,m+1}$ corresponds to exactly one order of the $a_{i}'s$. By the definition of $a_i$ this order corresponds to $w$. 
Now we can find the coefficients of $f-\sum _{|w|=m}F_w w=\sum _{|w|< m}F_w w$  by the induction hypothesis on $M_{m}(\F)\subset M_{m+1}(\F)$.
\end{proof}

\begin{proposition}\label{unique}
Suppose that a $G$-free map $f$ 
has a power series expansion in a neighbourhood $\cB(0,\delta)$ of $0$, $\delta=(\delta_n)_{n\in\N}$; 
i.e., 
$$f(X)=\sum_{m=0}^\infty\sum _{|w|=m}F_w w(X),$$
for $X\in\cB(0,\delta).$  
Then $F_w$ for $|w|=m$ is determined by the $m$-th derivative of the function $t\mapsto f[m+1](tX)$ at $0$ and hence by its evaluation  on  $M_{m+1}(\F)$.
\end{proposition}

\begin{proof}
Let $|t|<1$, then $tX\in \cB(0,\delta)[n]$ for every $X\in \cB(0,\delta)[n]$, and 
$$f[n](tX)=\sum_{m=0}^\infty t^m f_m[n](X)$$ 
is a convergent power series in $t$, where  $f_{m}$ are homogeneous free polynomials of degree $m$. 
We can thus determine $f_m[n](X)$ as 
$$\frac{1}{m!}\frac{\rm d}{{\rm d}t^m}f[n](tX)\bigg|_{t=0}.$$ 
Since $M_n(\F)$ does not admit a nontrivial polynomial identity (with involution) of degree $<n$ 
(see e.g.~\cite[Lemma 1.4.3, Remark 2.5.14]{Row}), $f_{m}$ is uniquely determined on $M_{m+1}(\F)$.  Hence we can recover $f_{m}$ by the $m$-th derivative of the function $t\mapsto f[m+1](tX)$.  
The coefficients of the polynomial $f_{m}$ can be constructively determined by evaluations on $M_{m+1}(\F)$ by Lemma
\ref{matenote}.
\end{proof}

\section{\RPS and Power Series Expansions about Non-scalar Points}\label{sec4}

{Theorem \ref{analit} gives a convergent power series expansion of a free analytic 
map about a \emph{scalar point} $a\in\F^g$. In this section we present power series expansions
about non-scalar points $A\in M_n(\F)^g$, whose homogeneous components are generalized
polynomials.}  
These are the topic of Subsection \ref{ssec4} and their obtained properties will be used in Subsection \ref{s2sec4} to deduce the desired power series expansion. 
Our methods are algebraic, and work for $G=\GL$ and $G=\OO$.
For $G=\GL$ a similar result has been obtained earlier  in \cite{KV} with a different proof.

Throughout this section $G\in\{\GL,\OO\}$.

\subsection{\RPS}\label{ssec4}
We call the elements of the free product $M_n(\F)\ast \F\X$ {\bf \Rps} (cf. \cite{Ami}, \cite[Section 4.4]{BMM}).
They can be written in the form
$$\sum a_{i_0}x_{k_1}a_{i_1}x_{k_2}\cdots a_{i_{\ell-1}}x_{k_\ell}a_{i_{\ell}},$$
where $a_{i_j}\in M_n(\F)$.
Let $e_{ij}$ denote the  standard matrix units of $M_n(\F)$. Then a basis of  $M_n(\F)\ast \F\X$ consists of monomials
$$ e_{i_0,j_0}x_{k_1}e_{i_1,j_1}x_{k_2}\cdots e_{i_{\ell-1},j_{\ell-1}}x_{k_\ell}e_{i_{\ell},j_{\ell}}$$
for  $\ell\in\N_0$, $I,J\in \{1,\dots,n\}^{\ell+1}$, $K\in \{1,\dots,g\}^\ell$, where $I=(i_0,\dots,i_\ell),J=(j_0,\dots,j_\ell),K=(k_1,\dots,k_\ell)$. 

The algebra  $M_n(\F)\ast \F\X$ can be evaluated (as an algebra with unity) in $M_{ns}(\F)$ for $s\in \N$ and we have an isomorphism
\beq\label{redfun}
{\rm Hom}_{M_n}(M_n(\F)\ast \F\X,M_{ns}(\F))\cong {\rm Hom}(\w_n(\F\X),M_{s}(\F)),
\eeq
where $\w_n$ denotes the matrix reduction functor (see \cite[Section 1.7]{Coh}). 
The isomorphism is a consequence of the identity
\beq\label{redfunid}
M_n(\F)\ast \F\X \cong M_n(\w_n(\F\X)).
\eeq
For the free algebra $\F\X=\F\langle x_1,\dots,x_g\rangle$ we have 
$$\w_n(\F\X)=\F\langle y^{(k)}_{ij}\mid 1\leq i,j\leq n,1\leq k\leq g\rangle,$$
where $y_{ij}^{(k)}$, as the brackets suggest, denote free noncommutative variables. 
For example, the evaluation of the element 
$$e_{11}x_1e_{12}x_2e_{22}\in M_2(\F)\ast\F\X $$
in $M_4(\F)$, defined by mapping $x_1,x_2$ to $A,B\in M_4(\F)$,  is 
$$
\begin{pmatrix}
\multicolumn{2}{c}{\multirow{2}{*}{$\;I_2\;$}}&\multicolumn{2}{c}{\multirow{2}{*}{\quad}}\\
&&&\\
\multicolumn{2}{c}{\multirow{2}{*}{\quad}}&\multicolumn{2}{c}{\multirow{2}{*}{$\; \;$}}\\
&&&\\
\end{pmatrix}
\begin{pmatrix}
\multicolumn{2}{c}{\multirow{2}{*}{$A_{11}$}}&\multicolumn{2}{c}{\multirow{2}{*}{$A_{12}$}}\\
&&&\\
\multicolumn{2}{c}{\multirow{2}{*}{$A_{21}$}}&\multicolumn{2}{c}{\multirow{2}{*}{$A_{22}$}}\\
&&&\\
\end{pmatrix}
\begin{pmatrix}
\multicolumn{2}{c}{\multirow{2}{*}{$\;\;$}}&\multicolumn{2}{c}{\multirow{2}{*}{$\;I_2\;$}}\\
&&&\\
\multicolumn{2}{c}{\multirow{2}{*}{\quad}}&\multicolumn{2}{c}{\multirow{2}{*}{$\;\;$}}\\
&&&\\
\end{pmatrix}
\begin{pmatrix}
\multicolumn{2}{c}{\multirow{2}*{$B_{11}$}}&\multicolumn{2}{c}{\multirow{2}{*}{$B_{12}$}}\\
&&&\\
\multicolumn{2}{c}{\multirow{2}{*}{$B_{21}$}}&\multicolumn{2}{c}{\multirow{2}{*}{$B_{22}$}}\\
&&&\\
\end{pmatrix}
\begin{pmatrix}
\multicolumn{2}{c}{\multirow{2}{*}{$\; \;$}}&\multicolumn{2}{c}{\multirow{2}{*}{\quad}}\\
&&&\\
\multicolumn{2}{c}{\multirow{2}{*}{\quad}}&\multicolumn{2}{c}{\multirow{2}{*}{$\;I_2\;$}}\\
&&&\\
\end{pmatrix}
=
\begin{pmatrix}
&&\multicolumn{2}{c}{\multirow{2}{*}{$A_{11}B_{22}$}}\\
&&&\\
&&&\\
&&&\\
\end{pmatrix} ,
$$
where $I_2$ denotes the identity of $M_2(\F)$, and $A_{ij}$ (resp. $B_{ij}$) denotes the $(i,j)$-block entry of $A$ (resp. $B$),  or 
$$(e_{11}\tnz I_2)A(e_{12}\tnz I_2)B(e_{22}\tnz I_2)=e_{12}\tnz A_{11}B_{22},$$
viewed as en element in $M_2(\F)\tnz M_2(\F)\cong M_4(\F)$.
 
Note that \eqref{redfun} and \eqref{redfunid} imply that no \Rp   vanishes on $M_{ns}(\F)$ for all $s$. In fact, two \Rps of degree $2d$ which agree on $M_{ns}(\F)$ for {some} $s>d$ are equal.  
We denote by ${\rm g}\cT_{ns}$ the ideal of  the elements in $M_n(\F)\ast\F\X$ that vanish when evaluated on $M_{ns}(\F)$ and let 
\[
C_{ns}=\F\big[x_{ij}^{(k)}\mid 1\leq i,j\leq ns,1\leq k\leq g\big].
\] 
The quotient algebra $\GMn=\big(\nA\big)/{\rm g}\cT_{ns}$ is isomorphic to the image of 
$$\phi:\nA\to M_{ns}(C_{ns}),$$ 
defined by mapping $x_k$ to the corresponding generic matrix $(x_{ij}^{(k)}).$ 
We write $\Rns$ for the  subalgebra of $M_{ns}(C_{ns})$ generated by $\GMn$ and traces of the elements in $\GMn$.
Note that every  polynomial map $p:M_{ns}(\F)^g\to M_{ns}(\F)$  can be considered as an element  $\td p\in M_{ns}(C_{ns})$. 

Let  $\GL_{ns}$ act on $M_{ns}(\F)$ by conjugation. 
We will be interested in the action of its subgroup $I_n\tnz \GL_s$.  In the next proposition we describe the invariants and concomitants of this action.

\begin{proposition}\label{partinv}
If $p: M_{ns}(\F)^g\to M_{ns}(\F)$ 
is an $I_n\otimes \GL_s$-concomitant, then $\td p\in \Rns$.
\end{proposition}

\begin{proof}
We can assume that $p$ is multilinear of degree $d$. 
Then $p$ corresponds to an element in 
$(M_{ns}(\F)^{\tnz d})^*\tnz M_{ns}(\F)$, which is canonically isomorphic to 
$M_n(\F)^{\tnz {d+1}}\tnz (M_s(\F)^{\tnz d})^*\tnz M_s(\F)$ as $I_n\otimes \GL_s$-module.  
The action of the group $I_n\tnz\GL_s$ reduces to the %diagonal 
action of $\GL_s$ on $(M_s(\F)^{\tnz d})^*\tnz M_s(\F)$. 
The invariants of this action correspond to multilinear trace polynomials of degree $d$ in $M_s(C_s)$ by \cite[Theorem 2.1]{Pro}. 
Moreover, the elements of the form 
$$\sum_{I,J}e_{i_1j_1}\tnz\cdots\tnz e_{i_dj_d}\tnz \tau_{IJ},$$ 
where $\tau_{IJ}\in (M_s(\F)^{\tnz d})^*\tnz M_s(\F)$ is a $\GL_s$-concomitant map, can be identified with multilinear elements of degree $d$ in $\Rns$. \end{proof}

\subsubsection{Generalized Polynomials with Involution}
To consider the case of algebras with involution we need to introduce some additional notation. 
We call the elements of the algebra  $M_n(\F)\ast\F\Xt$ {\em \Rps with involution}. 
By ${\rm g}\cT^{\dagger}_{ns}$ we denote the ideal of elements in $M_n(\F)\ast\F\Xt$ that vanish on $M_{ns}(\F)$. 
The quotient algebra is isomorphic to the subalgebra $\GMn^\dagger$  of $M_{ns}(C_{ns})$ generated by $\GMn$ and transposes of  elements in $\GMn$.
We write $\Rns^\dagger$ for the  subalgebra of $M_{ns}(C_{ns})$ generated by $\GMn^\dagger$ and traces of elements in $\GMn^\dagger$.

We have the (usual) action of $\OO_{ns}$ on $M_{ns}(C_{ns})$. 
The following proposition is the analog of Proposition \ref{partinv} for the action of $I_n\otimes \OO_s$ on $M_{ns}(C_{ns})$.

\begin{proposition}\label{partinvO}
If $p\in M_{ns}(\F)^g\to M_{ns}(\F)$ is an $I_n\otimes \OO_s$-concomitant, then $\td p\in \Rns^\dagger$.
\end{proposition}

\begin{proof}
The proof goes along the same lines as that of Proposition \ref{partinv}, we only need to invoke \cite[Theorem 7.2]{Pro} instead of \cite[Theorem 2.1]{Pro}.
\end{proof}

\subsubsection{Block ad centralizing $G$-concomitants}
Let us denote $\cM_n(\F)^k=\bigcup_s M_{ns}(\F)^k$, $k\in \N$. 
We say that a map $f:\cM_{n}(\F)^g\to \cM_{n}(\F)$ is $I_n\tnz G$-concomitant if 
$$f[ns]:\big(M_{n}(\F)\tnz M_s(\F)\big)^g\to M_{n}(\F)\tnz M_s (\F)$$ 
is a $I_n\tnz G_s$-concomitant for every $s\in \N$. 

\begin{proposition}\label{freeRum}
If $f:\cM_{n}(\F)^g\to \cM_{n}(\F)$ 
is a homogeneous polynomial map of degree $d$ and $I_n\tnz \GL$-concomitant $($resp. $I_n\tnz \OO$-concomitant$)$   
that preserves direct sums, then $f\in \nA$ $($resp. $f\in \nAt)$. 
\end{proposition}

\begin{proof}
We prove the lemma only in the case $G=\GL$, the modifications needed to treat the case $G=\OO$ are minor. 
We can assume that $f$ is multilinear. 
Since $f[ns]$ is a $I_n\tnz \GL_s$-concomitant, $f[ns]\in \Rns$ {by Proposition \ref{partinv}}. 
We can view $f[ns]$ as an element in $M_n(\F)^{\tnz d+1}\tnz (M_s(\F)^{\tnz d})^*\tnz M_s(\F)$ and write it in the form
$$f[ns]=\sum_{I,J}e_{i_1j_1}\tnz\cdots\tnz e_{i_dj_d}\tnz e_{i_{d+1}j_{d+1}}\tnz \tau^{(s)}_{IJ},$$
where $\tau^{(s)}_{IJ}$ is a $\GL_s$-concomitant. 
Let $s>d$. 
Since $f$ preserves direct sums we have
$$f[ns](X)\oplus f[ns](Y)=f[2ns](X\oplus Y).$$ 
We obtain for all $I,J$ an identity 
\beq\label{n2n}
\tau^{(s)}_{IJ}(X)\oplus \tau^{(s)}_{IJ}(Y)=\tau^{(2s)}_{IJ}(X\oplus Y).
\eeq
Let us fix $I,J$. To simplify the notation we write $\tau^{(s)}$ instead of $\tau^{(s)}_{IJ}$. 
We have 
$$\tau^{(s)}=\sum_M h_M^{(s)} M,$$ 
where $h_M$ is a pure trace polynomial, $M$ is a monomial in the variables $x_k$, and $\deg M+\deg h_M=d$. 
Then the identity \eqref{n2n} together with the fact that there are no trace identities of $M_s(\F)$ of degree $<s$ yields
$$h_M^{(s)}(X)=h_M^{(2s)}(X\oplus Y)=h_M^{(s)}(Y)$$ 
for all monomials $M$, which implies that \[\tau^{(s)}=\sum_M\alpha_M M\] for some $\alpha_M\in \F$. 
Thus, $f[ns]\in \GMn$ for every $s>d$ is represented by the same \Rp $\td f$. 
Since $f$ respects direct sums, we can identify it with $\td f$.  
\end{proof}

For a subset $B$ of $M_n(\F)$ we denote by $C(B)$ its {\em centralizer}  in $M_n(\F)$; i.e., 
\[
C(B)=\{c\in M_n(\F)\mid cb=bc \text{ for all $b\in B$}\},
\]
while  $C_{G_n}(B)$ stands for $ C(B)\cap G_n$. 
We say that a map $f: \cM_n(\F)^g\to \cM_n(\F)$ is a {\em \cg} if 
$f[ns]$ is a $(C_{G_n}(B)\tnz M_s(\F))\cap G_{ns}$-concomitant for every $s\in \N$.

\begin{lemma}\label{centGL}
Let $B$ be a subalgebra of $M_n(\F)$. 
If $f: \cM_n(\F)^g\to \cM_n(\F)$ is a homogeneous polynomial map of degree $d$ that is a \cgl, then $f\in C(C(B))\ast \F\X$.
\end{lemma}

\begin{proof}
By Lemma \ref{freeRum}, $f\in M_n(\F)\ast \F\X$. 
Since $\GL_n$ is dense in $M_n(\F)$, the vector space spanned by $C_{\GL_n}(B)$ coincides with $C(B)$. 
Thus we can choose a basis $\{c_1,\dots,c_t\}$ of $C(B)$ with $c_\ell\in \GL_n$. 
Let  $\{b_1,\dots,b_u\}$ be a basis of $C(C(B))$ and complete it to a basis $\{b_\ell\mid1\leq \ell\leq n^2\}$ of $M_n(\F)$. 
We can write $f$ uniquely as
$$f=\sum_{I,K} \alpha_{IK} b_{i_1}x_{k_1} b_{i_2}\cdots x_{k_d} b_{i_{d+1}},$$
where $I$ runs over all $d+1$-tuples of elements in $\{1,\dots,n^2\}$, and $K$ over all $d$-tuples of elements in $\{1,\dots,g\}$. 
Take $s>d$ and evaluate $f$ on $M_{2nts}(\F)\cong M_n(\F)\otimes M_{2t}(\F)\otimes M_s(\F)$. 
Note that $f$ on $M_{2nts}(\F)$ can be identified with  the evaluation of the generalized polynomial 
\[
\sum_{I,K} \alpha_{IK} \Big(\sum_{i=1}^{2t}b_{i_1}\tnz e_{ii}\Big)x_{k_1} \Big(\sum_{i=1}^{2t}b_{i_2}\tnz e_{ii}\Big)\cdots x_{k_d} \Big(\sum_{i=1}^{2t}b_{i_{d+1}}\tnz e_{ii}\Big).
\]
in $M_{2nt}(\F)\ast F\X=(M_n(\F)\tnz M_{2t}(\F))\ast \F\X$, 
and every element in $M_{2nt}(\F)\ast \F\X$ has a unique expression with the matrix coefficients $b_\ell\tnz e_{ij}$, $1\leq i,j\leq 2t$, $1\leq \ell\leq n^2$, on $M_{2nts}(\F)$ as $s>d$. 
Let 
\[
\s=\Big(\alpha1\tnz 1+\beta\sum_{\ell=1}^t(c_\ell\tnz e_{\ell,t+\ell}-c_\ell^{-1}\tnz e_{t+\ell,\ell})\Big)\tnz 1\in (C_{\GL_n}(B)\tnz M_{2t}(\F)\tnz M_s(\F))\cap \GL_{2nts}
\] 
for $\alpha^2+\beta^2=1$, $\alpha,\beta\in \RR$. 
Note that 
\beq\label{centelem}
\s^{-1}=\Big(\alpha1\tnz 1-\beta\sum_{\ell=1}^t(c_\ell\tnz e_{\ell,t+\ell}-c_\ell^{-1}\tnz e_{t+\ell,\ell})\Big)\tnz 1.
\eeq 
Since $f$ is a \cgl  we have 
\[
\sum_{I,K} \alpha_{IK} b_{i_1}^\s x_{k_1} b_{i_2}^\s\cdots x_{k_d} b_{i_{d+1}}^\s
=\sum_{I,K} \alpha_{IK} b_{i_1}x_{k_1} b_{i_2}\cdots x_{k_d} b_{i_{d+1}},
\]
where by a slight abuse of notation $b_{i}$ denotes $b_i\tnz 1\tnz 1$, and 
\begin{multline}\label{eq:no1}
b_i^\s=\s^{-1}b_i\s= \alpha^2b_{i}\tnz 1\tnz 1+\sum_{\ell=1}^t\beta^2c_\ell b_{i}c_\ell^{-1}\tnz e_{\ell\ell}\tnz 1+\beta^2c_\ell^{-1}b_{i}c_\ell\tnz e_{t+\ell,t+\ell}\tnz 1\\
+\alpha\beta( b_{i}c_\ell-c_\ell b_{i})\tnz e_{\ell,t+\ell}\tnz 1-\alpha\beta(b_i c_\ell^{-1}-c_\ell^{-1}b_i)\tnz e_{t+\ell,\ell}\tnz 1.
\end{multline}
Since $s>d$ both sides of equation \eqref{eq:no1} have a unique expression as generalized polynomials in $M_{2tn}\ast \F\X$ 
with the generalized coefficients $b_\ell\tnz e_{ij}$, $1\leq i,j\leq 2t$, $1\leq \ell\leq n^2$.  
We thus derive 
\beq\label{eq:no2}
\sum_k \alpha_{I^j_kK}(b_kc_\ell-c_\ell b_k)=0
\eeq
for every $1\leq j\leq d+1$, $1\leq \ell\leq t$, where $I^j_k$ denotes a tuple of $d+1$-elements in $\{1,\dots,n^2\}$ with $k$ at the $j$-th position. 
Equation \eqref{eq:no2} implies that 
\[
\sum_k  \alpha_{I^j_kK}b_k\in C(C(B)),
\]
 which is by the choice of $b_\ell$, $1\leq \ell\leq n^2$, only possible if $\alpha_{I^j_kK}=0$ for $b_k\not\in C(C(B))$. Therefore we have $f\in C(C(B))\ast \F\X$.
\end{proof}

\begin{lemma}\label{preCentO}
If $B$ is a $*$-subalgebra of $M_n(\RR)$, 
then the subalgebra generated by $C_{\OO_n}(B)$ is equal to $C(B)$, and $C(C_{\OO_n}(B))=C(C(B))=B$.
\end{lemma}

\begin{proof}
Since $B$ is a $*$-subalgebra of $M_n(\RR)$, $C(B)$ is also a $*$-subalgebra of $M_n(\RR)$, thus semisimple. 
Notice that in order to show that $\RR\langle C_{\OO_n}(B)\rangle$, the subalgebra of $C(B)$ generated by $C_{\OO_n}(B)$, coincides with $C(B)$, 
%Notice that it is enough to show this for all simple components of $C(B)$, 
we can assume that $C(B)$ is simple. 
We have $c^t-c\in {\rm span}\, C_{\OO_n}(B)$, the vector subspace of $M_n(\RR)$ spanned by $C_{\OO_n}(B)$, for every $c\in C(B)$. 
Indeed, $e^{\lambda (c^t-c)}\in C_{\OO_n}(B)$ for every $\lambda\in \RR$, $c\in C(B)$ yields $c^t-c\in {\rm span}\,C_{\OO_n}(B)$. 
If $C(B)$ %only involves simple components isomorphic 
is isomorphic to  $\RR$, $M_2(\RR)$, $\CC$, or $M_2(\CC)$, where the involution on $\CC$ is the complex conjugation, 
then one can easily verify that ${\rm span}\,C_{O_n}(B)=C(B)$. 
Recall that a finite dimensional %semi
simple  $\RR$-algebra with involution    
%not isomorphic to 
%$\RR^{k_1}\oplus M_2(\RR)^{k_2}\oplus  \CC^{k_3}\oplus M_2(\CC)$
%for some $k_1,k_2,k_3,k_4\in \N$, 
which is not
%does not include a simple component 
isomorphic to $\RR$, $M_2(\RR)$, $\CC$, or $M_2(\CC)$ 
coincides with
its subalgebra generated by the skew-symmetric elements (see e.g.~\cite[Lemma 2.26]{KMRT}). 
Therefore $\RR\langle C_{\OO_n}(B)\rangle=C(B)$, which further implies $C(C_{\OO_n}(B))=C(C(B))$, and
the identity $C(C(B))=B$ follows from the double centralizer theorem (see e.g.~\cite[Theorem 1.5]{KMRT}).
%von Neumann double commutant theorem (see e.g.~\cite[Theorem 4.3.8]{Li}).
\end{proof}

\begin{lemma}\label{centO}
Let $B$ be a $*$-subalgebra of $M_n(\RR)$. 
If $f: \cM_n(\RR)^g\to \cM_n(\RR)$ is a homogeneous polynomial map of degree $d$ that is a \co, then $f\in B\ast \RR\Xt$.
\end{lemma}

\begin{proof}
Since the proof is similar to that of Lemma \ref{centO} we omit some of the details. 
By Proposition \ref{partinvO} we have $f\in M_n(\RR)\ast \RR\Xt$. 
Let $c_1,\dots,c_t$ be a basis of ${\rm span} \,C_{\OO_n}(B)$, the vector space spanned by $C_{\OO_n}(B)$, with $c_\ell\in \OO_n$. 
%We can follow steps in the proof of Lemma \ref{centGL} as $\s$, defined in \eqref{centelem}, belongs to $(C_{\OO_n}\tnz \cM(\RR))\cap \OO$, and  
%Lemma \ref{preCentO} enables us to carry out also the last part of the proof. write $f$ uniquely as
Let us write
$$f=\sum_{I,K} \alpha_{IK} b_{i_1}u_{k_1} b_{i_2}\cdots u_{k_d} b_{i_{d+1}},$$
where $u_k\in \{x_k,x_k^t\}$. 
Take $s>d$ and evaluate $f$ on $M_{2nts}(\F)$. 
Let 
\[
\s=\Big(\alpha1\tnz 1+\beta\sum_{\ell=1}^t(c_\ell\tnz e_{\ell,t+\ell}-c_\ell^t\tnz e_{t+\ell,\ell})\Big)\tnz 1\in (C_{\OO_n}(B)\tnz M_{2t}(\F)\tnz M_s(\F))\cap \OO_{2nts}
\] 
for $\alpha^2+\beta^2=1$, $\alpha,\beta\in \RR$. 
Note that $\s\in \OO_{2nts}$ and
\beq\label{centelem'}
\s^t=\Big(\alpha1\tnz 1-\beta\sum_{\ell=1}^t(c_\ell\tnz e_{\ell,t+\ell}-c_\ell^t\tnz e_{t+\ell,\ell})\Big)\tnz 1.
\eeq 
Since $f$ is a \co  we have 
\[
\sum_{I,K} \alpha_{IK} b_{i_1}^\s u_{k_1} b_{i_2}^\s\cdots u_{k_d} b_{i_{d+1}}^\s
=\sum_{I,K} \alpha_{IK} b_{i_1}u_{k_1} b_{i_2}\cdots u_{k_d} b_{i_{d+1}},
\]
where $b_{i}$ denotes $b_i\tnz 1\tnz 1$, and 
\begin{multline*}
b_i^\s=\s^{t}b_i\s= \alpha^2b_{i}\tnz 1\tnz 1+\sum_{\ell=1}^t\beta^2c_\ell b_{i}c_\ell^{t}\tnz e_{\ell\ell}\tnz 1+\beta^2c_\ell^{t}b_{i}c_\ell\tnz e_{t+\ell,t+\ell}\tnz 1+\\
+\alpha\beta( b_{i}c_\ell-c_\ell b_{i})\tnz e_{\ell,t+\ell}\tnz 1-\alpha\beta(b_i c_\ell^{t}-c_\ell^{t}b_i)\tnz e_{t+\ell,\ell}\tnz 1.
\end{multline*}
As $s>d$ both sides of the last identity have a unique expression as generalized polynomials in $M_{2tn}\ast \RR\Xt$ 
with the generalized coefficients $b_\ell\tnz e_{ij}$, $1\leq i,j\leq 2t$, $1\leq \ell\leq n^2$.   
Thus, $\alpha_{I^j_kK}=0$ for $b_k\not\in C(C_{\OO_n}(B))$, where $I^j_k$ denotes a tuple of $d+1$-elements in $\{1,\dots,n^2\}$ with $k$ at the $j$-th position. 
Since $C(C_{\OO_n}(B))=B$ by Lemma \ref{preCentO}, $f$ belongs to $B\ast \RR\Xt$.
\end{proof}

\subsection{Power Series Expansions about Non-Scalar Points}\label{s2sec4}

{We next turn to analytic free maps and exhibit
their power series expansions about a non-scalar point $A$. Homogeneous components of such an expansion will be \Rps.} For $G=\GL$ their matrix coefficients belong to the double centralizer $C(C(A))$, while for $G=\OO$ they lie in the $*$-subalgebra $\F\langle A,A^t\rangle$ generated by $A$.

Let us first  introduce  neighbourhoods of non-scalar points. 
Given $A\in M_n(\F)^g$, set
$$\cB(A,\delta)=\bigcup_{s=1}^\infty\big\{X\in M_{ns}(\F)^g\mid \big\|X-\bigoplus_{i=1}^s A\big\|<\delta_s\big\},$$
where $\delta=(\delta_s)_{s\in \N}$, $\delta_s>0$ for every $s\in \N$.

\subsubsection{$GL$-free maps}\label{s1s2sec4}
The next theorem gives a power series expansion of a $\GL$-free map $f$ about $A=(A_1,\dots,A_g)\in M_n(\F)^g$,  whose
matrix coefficients are elements of the double centralizer
algebra $C(C(\F \langle A\rangle))\subseteq M_n(\F)$ of the subalgebra $\F\langle A\rangle$ generated by $A_1,\dots,A_g$.
\begin{theorem}\label{rumanalit}
Let $\cU$ be a $\GL$-free set, $f:\cU\to\cM(\F)$ be an $\F$-analytic $\GL$-free  map, and let  $\cB(A,\delta)\subseteq \cU$, where $A\in M_n(F)^g$, and $\delta=(\delta_s)_{s\in\N}$, $\delta_s>0$ for every $s\in \N$. 
Then there exist unique generalized polynomials  $f_m\in C(C(\F\langle A))\rangle\ast \F\X$  of degree $m$  
so that the formal power series 
\beq\label{eq:ps}
F(X)=\sum_{m=0}^\infty f_m(X-A),
\eeq
converges in norm on the neighbourhood $\cB(A,\delta)$ of $A$  to $f$. 
\end{theorem}

\begin{proof}
As $A\in \cU[n]$ and $\cU$ is a $\GL$-free set we have 
\[A^{\oplus s}=\bigoplus_{i=1}^s A\in \cU[ns]\] for every $s\in \N$. 
Since $f[ns]$ is analytic in a neighbourhood of $A^{\oplus s}$, the function 
$$t\mapsto f[ns]\Big(A^{\oplus s}+t\big(X-A^{\oplus s}\big)\Big)$$
is defined and analytic for all $|t|<\delta_X$, where $\delta_X$ depends on $X\in M_{ns}(\F)$. 
Thus, we can expand it in a power series
\beq\label{homt} 
f[ns]\Big(A^{\oplus s}+t\big(X-A^{\oplus s}\big)\Big)=\sum_{m=0}^\infty t^mf[ns]_{m}\big(X-A^{\oplus s}\big)
\eeq
that converges for $|t|<\delta_X$. 
If $X\in \cB(A,\delta)$, then we have $\delta_X\geq1$. 
We claim that 
$f[ns]_{m}$ is a homogeneous polynomial function of degree $m$. 
Indeed, as
$$\sum_{m=0}^\infty t_1^mf[ns]_{m}\Big (t_2\big(X-A^{\oplus s}\big)\Big) =f[ns]\Big(A^{\oplus s}+t_1t_2\big(X-A^{\oplus s}\big)\Big)=\sum_{m=0}^\infty t_1^mt_2^mf[ns]_{m}\big(X-A^{\oplus s}\big)$$
for all $t_1$  that satisfy $|t_1|, |t_1t_2|<\delta_X$, we obtain \[f[ns]_{m}(tY)=t^mf[ns]_m(Y)\] for all $t\in \F$, $Y\in M_{ns}(\F)^g$. 
Let us show that 
$$f_{m}:\cM_{n}(\F)^g\to \cM_{n}(\F)$$ 
defined by 
$f_{m}[ns]:=f[ns]_{m}$ 
is a \cgl %(resp. \co)  
 that preserves direct sums. 
Take $s\in \N$, $\sigma\in (C_{\GL_n}(F\langle A\rangle)\tnz M_s(\F))\cap \GL_{ns}$ %(resp. $(C_{\OO_n}(F\langle A\rangle)\tnz M_s(\F))\cap \OO_{ns}$)
 and note that 
\[\s A^{\oplus s}\s^{-1}=A^{\oplus s}.\] 
Then the identity 
\begin{eqnarray*}
\sum t^m \s f[ns]_{m}\big(X-A^{\oplus s}\big) \s^{-1}&=& \s f[ns]\Big(A^{\oplus s}+t\big(X-A^{\oplus s}\big)\Big) \s^{-1}\\
&=&f[ns]\Big(A^{\oplus s}+t\big(\s X\s^{-1}-A^{\oplus s}\big)\Big)\\
&=&\sum t^m f[ns]_{m}\Big(\s\big( X-A^{\oplus s}\big)\s^{-1}\Big),
\end{eqnarray*}
for all  small enough $t$
yields the desired conclusion. 

To conclude the proof of the existence we proceed as at the end of the proof of existence in %The second part {of the theorem} follows as the corresponding part
% in the proof of
 Theorem \ref{analit}. 
Thus, $f_{m}\in C(C(\F\langle A))\rangle\ast \F\X$ %(resp. $f_m\in\F\langle A,A^t\rangle\ast\F\Xt$) 
by Lemma \ref{centGL}. % (resp. by Lemma \ref{centO}). 
Note that setting $t=1$ in \eqref{homt} establishes the existence of the desired power series.

For the uniqueness, we can also follow the proof of uniqueness in Theorem \ref{analit} carried out in Lemma \ref{matenote} and Proposition \ref{unique}, after recalling the identity \eqref{redfun}. 
Hence we can recover $f_m$ by the $m$-th derivative of the function $t\mapsto f[n(m+1)](t(X-A))$ at $0$, and 
the matrix coefficients of the generalized polynomial $f_m$ can be determined by evaluations on $M_{n(m+1)}(\F)$.
\end{proof}

\begin{remark}
If $f$ is a uniformly bounded $\GL$-free map  then the convergence of $F$ in
\eqref{eq:ps} is uniform, which can be proved in the same way as the analogous statement for $\F=\CC$ and power series expansion about scalar points in the last part of the proof of \cite[Proposition 2.24]{HKM12}. 
The only modification needed is to replace $\exp({\mathbbm{i}t})I_{ns},\exp(-\mathbbm{i}mt)I_{ns}\in M_{ns}(\CC)$ in the equation
\[
C\geq \Big\|\frac{1}{2\pi}\int f(\exp(\mathbbm{i}t)X)\exp(-\mathbbm{i}mt)dt\Big\|=\|f^{(m)}(X)\|
\]
 with the corresponding matrices in $M_{2ns}(\RR)$. 
\end{remark}

In general one cannot expect the matrix coefficients of the power series expansion of a $\GL$-free map $f$ about a non-scalar point $A$ to lie in $\F\langle A\rangle\ast\F\X$.  
In this case one would have $f(A)\in \F\langle A\rangle$,  which is not always the case by \cite[Theorem 7.7]{AM2}. 
However, this does hold true in the case that $A$ is a generic point. 
That is, if $g=1$, then $A$ is similar to a diagonal matrix with $n$ distinct eigenvalues, and   
 if $g>1$ then $\F\langle A\rangle=M_n(\F)$.

\begin{corollary}
Let $\cU$ be a $\GL$-free set, $f:\cU\to\cM(\F)$ be an $\F$-analytic $\GL$-free  map, and let  $\cB(A,\delta)\subseteq \cU$, where $A\in M_n(F)^g$ is a generic point, and $\delta=(\delta_s)_{s\in\N}$, $\delta_s>0$ for every $s\in \N$. 
Then there exist generalized polynomials $f_m\in \nA$ of degree $m$  
so that the formal power series 
\[
F(X)=\sum_{m=0}^\infty f_m(X-A),
\]
converges in norm on the neighbourhood $\cB(A,\delta)$ of $A$  to $f$. 
\end{corollary}

\subsubsection{$\OO$-free maps}
In the case of free maps with involution the matrix coefficients in the power series expansion of an $\OO$-free map about $A=(A_1,\dots,A_g)\in M_n(\F)^g$ lie in the $*$-subalgebra $\F\langle A,A^t\rangle$ of $M_n(\F)$ generated by $A_1,\dots,A_g$. 
This contrasts the analogous result for $\GL$-free maps (Theorem \ref{rumanalit}) where the double centralizer of $\F\langle A\rangle$ is required. 
\begin{theorem}\label{rumanalitO}
Let $\cU$ be an $\OO$-free set, $f:\cU\to\cM(\F)$ be an $\F$-analytic $\OO$-free  map, and let  $\cB(A,\delta)\subseteq \cU$, where $A\in M_n(F)^g$, and $\delta=(\delta_s)_{s\in\N}$, $\delta_s>0$ for every $s\in \N$. 
Then there exist unique generalized polynomials $f_m\in\F\langle A,A^t\rangle\ast\F\Xt$ of degree $m$  
so that the formal power series 
\[
F(X)=\sum_{m=0}^\infty f_m(X-A),
\]
converges in norm on the neighbourhood $\cB(A,\delta)$ of $A$  to $f$. 
\end{theorem}

\begin{proof}
The proof resembles that of Theorem \ref{rumanalit} with obvious modifications. One only  needs to  apply Lemma \ref{centO} instead of Lemma \ref{centGL}.
\end{proof}
\section{Inverse Function Theorem for Free Maps}\label{sec5}

As an application of the tools and techniques developed we present an inverse and implicit function theorem for free maps. 
For $G=\GL$ these results have been obtained by Pascoe \cite{Pas}, Agler and McCarthy \cite{AM2},  
Kaliuzhnyi-Verbovetskyi and Vinnikov (private communication).

 Following \cite{KV} we recall two {topologies  on $\cM(\F)^g$.}
The first  is the {\bf finitely open topology}. Its basis are open sets $U$ such that the intersection of $U$ with $M_n(\F)^g$ is open for every $n\in \N$. The second topology is the {\bf uniformly open topology} and its basis consists of  sets of the form  
$${\mathcal B}(A,r)=\bigcup_{s=1}^\infty\big\{X\in M_{ns}(\F)^g\mid \big\|X-\bigoplus_{i=1}^s A\big\|<r\big\},$$ 
for $A\in M_n(\F)^g$, $n\in \N$, $r\geq 0$.
Further topologies in this free context are considered in \cite{AM1,AM2}.

Let us recall a version of the classical inverse function theorem, giving information on the injectivity domain (see e.g.~\cite[Theorem \rom{14}.1.2]{Lan},  \cite[Theorem 2.5.1]{KrPa}, \cite[Theorem 0.8.3]{KK}). We state it only in the case when $f:\cU\to V$ for $\cU\subset V$, $0$ is in the domain of $f$, $f(0)=0$, ${\rm D}f(0)=\id_V$,  
to which the general case can be reduced by replacing the function $f:\cU\to V$ with the function  $\ol{f}(x)={\rm D}f(x_0)^{-1}(f(x+x_0)-f(x_0))$, if $x_0$ is the point in the domain of $f$. 
Here $\rm D$ denotes the Fr\'echet derivative. We say that $f\in \cC^r$ if all ${\rm D}^kf$, $1\leq k\leq r$,
 exist and are continuous.

\begin{theorem}
\label{IFT}
Let $V$ be a Banach space, $\cU \subset V$ an open set containing $0$, $f:\cU\to V$, and let $f\in  \cC^r$ for some $r\in \N$ $($resp. $f$ is analytic$)$. 
Let ${\rm D}f(0):V\to V$ be a  continuous bijective linear map. 
If ${\rm Ball}(0,2\delta)\subseteq\cU$ and $\|{\rm D}(x-f(x))\|<\frac{1}{2}$  for $\|x\|<2\delta$, then $f$ is injective on ${\rm Ball}(0,\delta)$, and there exists $h:{\rm Ball}(0,\frac{\delta}{2})\to \cV$, where $\cV$ is an open subset of ${\rm Ball}(0,\delta)$, such that $hf=\id_\cV$, $fh=\id_{{\rm Ball}(0,\frac{\delta}{2})}$, and $h\in \cC^r$ $($resp. $h$ is analytic$)$. 
\end{theorem}

With a slight abuse of notation, we call a $g'$-tuple of $G$-free maps $f=(f_1,\dots,f_{g'})$, $f_i:\cU \to\cM(\F)$, also a $G$-free map. 
Throughout this section we let $G\in \{\GL,\OO\}$. 

\subsection{Uniformly Open Topology}
In this subsection we work with the uniformly open topology. The Fr\'echet derivative ${\rm D}f$ is  continuous in the uniformly open topology at $A\in M_n(\F)^g$ if for every $\varepsilon>0$ there exists $\delta >0$ such that $\|{\rm D}f(X)-{\rm D}f(A^{\oplus s})\|<\varepsilon$ if $s\in \N$ and $X\in \cB(A,\delta)[ns]$.

\begin{theorem}[Inverse free function theorem]\label{IFFT}
Let $\cU\subset \cM(\F)^g$ be an open $G$-free set containing $0$, $f:\cU\to \cM(\F)^{g'}$ a $G$-free map, and let $f\in \cC^r$  
 for $r\in \N$ $($resp. $f$ analytic$)$, with ${\rm D}f(0)$ invertible as a continuous linear map. 
Then there exist open $G$-free sets $\cW\subset \cM(\F)^{g}$, $\cW'\subset \cM(\F)^{g'}$ containing $0$,  $f(0)$  respectively, and a $G$-free map $h:\cW'\to \cW$ so that
$fh=\id_{\cW'}$, $hf=\id_{\cW}$, and $h\in \cC^{r}$ $($resp. $h$ analytic$)$. 
Moreover,  $h$ is analytic for every $r\in \N$ in the case $G=\GL$. 
\end{theorem}

\begin{proof}
Since $\cM(\F)^g$ is not a Banach space we cannot directly apply Theorem \ref{IFT}. 
However, we can use it levelwise. 
Without loss of generality we can assume that $g=g'$, $f(0)=0$ and ${\rm D}f(0)=\id_{\cM(\F)^g}$ by replacing $f$ with the function 
$$\ol f:\cM(\F)^g\to \cM(\F)^g,\quad \ol f={\rm D}f(0)^{-1}(f-f(0)).$$ 
As ${\rm D}f$ is continuous on $\cU$ and invertible at $0$ with a continuous inverse in the uniformly open topology, 
there exists (by the definition of the topology) $\delta>0$ such that $\cB(0,2\delta)\subseteq \cU$ and $\|{\rm D}(x-f(x))\|<\frac{1}{2}$ for $\|x\|<2\delta$.
 Theorem \ref{IFT} therefore implies that $f$ is injective on $\cB(0,\delta)$, and 
provides
 a $\cC^r$-map $h:\cB(0,\frac{\delta}{2})\to \cV$, where $\cV$ is an open subset of $\cB(0,\delta)$, that satisfies the desired identities. 

Let us first show that $\cV$ is an $\OO$-free set and $h$ is an $\OO$-free map. 
Let $u\in \OO_n$, $Y\in \cB(0,\frac{\delta}{2})[n]$. 
As $uYu^t\in \cB(0,\frac{\delta}{2})[n]$ and $f$ is a $G$-free map we have 
\beq\label{Ukon}
f\big(h(uYu^t)\big)=uYu^t=uf\big(h(Y)\big)u^t=f\big(uh(Y)u^t\big).
\eeq
Since $uh(Y)u^t\subset u\cV u^t\subset \cB(0,\delta)$ and $f$ is injective on $\cB(0,\delta)$, $h$ respects $\OO$-similarity. 
In the same way one can show that $h$ respects direct sums, so it is indeed an $\OO$-free map. 
In consequence, $\cV=h(\cB(0,\frac{\delta}{2}))$ is an $\OO$-free set. 
Thus, in the case $G=\OO$, the proposition follows. 

It remains to consider the case $G=\GL$. 
We claim that $h$ is analytic in this case. In the case $\F=\CC$, $f$ is analytic (see \cite[Proposition 2.5]{HKM11} or \cite[Theorem 7.2]{KV}). 
Our assumptions imply that $f$ is (uniformly) bounded in 
 $\cB(0,\delta)$, therefore we can apply \cite[Theorem 7.23, Remark 7.35]{KV} to deduce that $f$ is analytic also in the case $\F=\RR$. Thus, $h$ is analytic by Theorem \ref{IFT}. 
Since $h$ is an $\OO$-free map according to the previous paragraph, it can be expanded in a power series \eqref{eq:pw} in $x,x^t$ about $0$ by Theorem \ref{analit}, which converges in $\cB(0,\frac{\delta}{2})$. 
 Note that \eqref{Ukon} holds also if we replace $u,u^t$ by $\s,\s^{-1}$ respectively, for $\s\in \GL_n$ such that 
$\s Y\s^{-1}\in \cB(0,\frac{\delta}{2})$, $\s h(Y)\s^{-1}\in \cB(0,{\delta})$. 
 Note that for every $Y\in \cB(0,\frac{\delta}{2})$ there exists $\delta_\s>0$, 
such that $t\s Y\s^{-1}\in \cB(0,\frac{\delta}{2}),\s h(tY)\s^{-1}\in \cB(0,\delta)$ for every $|t|<\delta_\s$. Thus,
\[
h(\s tY\s^{-1})=\s h(tY)\s^{-1}
\]
for every $|t|<\delta_\s$. Writing this identity as a power series in $t$, we can deduce that each homogeneous part $h_m$ of the power series $H$ of $h$ is a $\GL$-concomitant. Thus, $H$ is a power series in $x$, and $h$ is a $\GL$-free map on $\cB(0,\frac{\delta}{2})$. 
Now notice that the $\GL$-similarity invariant envelopes 
\[
\cW=\wtd\cV,\quad\cW'=\wtd{\cB(0,\frac{\delta}{2})}
\]
 are open sets since the function $X\mapsto \s X\s^{-1}$ is an (analytic) isomorphism. 
As $\cU$ is a $G$-free set, $\cW$ is contained in $\cU$.  
Furthermore, $\td h$  (cf. Proposition \ref{simenv}) maps $\cW'$ to $\cW$. 
Thus, we only need to check that $ f$ and $\td h$ satisfy the desired identities. 
Let $\td X=\s X\s^{-1}\in \cW$, where $X\in \cV[n], \s\in \GL_n$. Then 
$$\td h\Big(f\big(\s X\s^{-1}\big)\Big)=\td h\big(\s f(X)\s^{-1}\big)=\s h(f(X))\s^{-1}=\s X\s^{-1}$$
implies that $\td h f=\id_{\cW}$. The identity $f\td h=\id_{\cW'}$ can be checked similarly. 
\end{proof}

The proof  used in the classical setting to derive the implicit function theorem from the inverse function theorem can be also
utilized in the free setting. Thus, we obtain an implicit free function theorem. 
We denote by ${\rm D}_2f(a,b)$, where  $f:\cU\times \cV\to \cW$, and $(a,b)\in \cU\times \cV$,  the Fr\'echet derivative of the function $y\mapsto f(a,y)$ evaluated at $b$.

\begin{corollary}[Implicit free function theorem]\label{IFFT2}
Let $\cU_1\times \cU_2\subseteq \cM(\F)^g\times \cM(\F)^{g'}$ be an open $G$-free set, $f:\cU_1\times \cU_2\to \cM(\F)^{g'}$ a $G$-free map, and let $f\in \cC^r$  
for some $r\in \N$, with ${\rm D_2}f(0,0)$ invertible. 
There exist an open $G$-free set $\cV_1\times\cV_2$ containing $(0,0)$, and a $G$-free map $h:\cV_1\to \cV_2 $, $h\in\cC^r$, such that
$f(x,y)=0$ for $(x,y)\in \cV_1\times \cV_2$ if and only if $y=h(x)$. 
\end{corollary}

We now turn our attention to the inverse function theorem about neighbourhoods of non-scalar points.
Let us denote 
$$C_{G}(A)=\{\s\in G_n\mid \s A_i=A_i\s, \,1\leq i\leq g\}$$
 for $A=(A_1,\dots,A_g)\in M_n(\F)^g$. 
We say that $\cU\subset \cM_n(\F)$ is a $C_G(A)\otimes G$-free set if it is closed under direct sums and simultaneous $C_G(A)\otimes G$-similarity.
 By 
$$\td\D f(A):\cM_n(\F)^g\to \cM_n(\F)^{g'}$$ 
for $f:\cU\to \cM_n(\F)^{g'}$, $A\in\cU\subseteq \cM_n(\F)^g$, we denote the linear map defined levelwise for every $s\in \N$ as
\[\td\D f(A)[ns](H):=\D f(A^{\oplus s})(H).\]
The next theorem generalizes Theorem \ref{IFFT} to the case of non-scalar center points. 

\begin{theorem}\label{IFFTn}
Let $\cU\subset \cM(\F)^g$ be an open $G$-free set, $A\in \cU[n]$, $f:\cU\to \cM(\F)^{g'}$ a $G$-free map, and let $f\in \cC^r$ 
 for $r\in \N$, with ${\td\D}f(A)$ invertible as a continuous linear map. 
There exist open $C_G(A)\otimes G$-free sets $\cW\subset \cM_n(\F)^{g}$, $\cW'\subset \cM_n(\F)^{g'}$ containing $A$,  $f(A)$ respectively, 
and a $C_G(A)\otimes G$-free map $h:\cW'\to \cW$ so that
$fh=\id_{\cW'}$, $hf=\id_{\cW}$, and $h\in \cC^{r}$. 
\end{theorem}

\begin{proof}
Note that 
\[\D f(\s X\s^{-1})(\s H\s^{-1})=\s \D f(X)(H)\s^{-1}\]
 for every $X,H\in M_n(\F)^g,\s\in G_n,n\in \N$. Since $A\in \cU$, which is an open $G$-free set, there exists $\delta>0$ such that $\cB(A,\delta)\subset \cU$. Then the function $\ol f:\cB(0,\delta)\cap\cM_n(\F)^g\to  \cM_n(\F)^g$ defined by 
\[{\ol f}[ns]:\cB(0,\delta)\cap M_{ns}(\F)^g\to M_{ns}(\F)^g,\quad {\ol f}[ns](X):=\D f\big (A^{\oplus s}\big)^{-1}\Big(f\big(X+A^{\oplus s}\big)-f\big(A^{\oplus s}\big)\Big )\] 
is $C_G(A)\otimes G$-free with ${\ol f}(0)=0$, $\D {\ol f}(0)=\id_{\cM_n(\F)}$. 
A similar reasoning to that in the proof of Theorem \ref{IFFT} with obvious modifications and using Theorem \ref{rumanalit} in the place of Theorem \ref{analit} now 
yields the desired conclusions.
\end{proof}

\subsection{Finitely Open Topology}
Now we state a weak form of the inverse function theorem for the finitely open topology. The  Fr\'echet derivative  ${\rm D}f$ is continuous in the finitely open topology if ${\rm D}f[n]$ is continuous for every $n\in \N$.

\begin{proposition}\label{foi}
Let $\cU\subseteq \cM(\F)^g$ be an open $G$-free set, $f:\cU\to \cM(\F)^{g'}$ a $G$-free map, and let $f\in \cC^r$ %in the finitely open topology 
for some  $r>0$ with ${\rm D}f(0)$ be invertible. 
There exist finitely open sets $\cW,\cV$, containing $0$, $f(0)$ respectively, and a free $\OO$-concomitant map $h:\cV\to \cW$ such that
$fh=\id_{\cV}$, $hf=\id_{\cW}$, and $h\in \cC^{r}$. In the case $\F=\CC$, $h$ is a  a free $G$-concomitant map.
\end{proposition}

\begin{proof}
By the classical inverse function theorem  we can find for every $n\in \N$ 
 neighbourhoods $\cV_n$, $\cB(0,\delta_n)$ of $0$, $f[n](0)$ respectively, such that 
$f[n]:\cV_n\to{\cB(0,\delta_n)}$ is a diffeomorphism with the inverse $h[n]\in \cC^r$. 
Since $\cB(0,\delta_n)$ is  $\OO_n$-invariant so is $\cV_n$ for every $n\in \N$. 
As in the proof of  Theorem \ref{IFFT} it is easy to show that $h(uYu^t)=uh(Y)u^t$ for every $u\in \OO_n$, $Y\in \cV_n$. 
By the definition of the finitely open topology, the sets $\cV=\bigcup_n \cV_n$, $\cW=\bigcup_n\cB(0,\delta_n)$ 
 are finitely open. This establishes the proposition in the case $G=\OO$. In the case $G=\GL_n$ and $\F=\CC$  we proceed as in the proof of Theorem \ref{IFFT}, and replace $\cV$, $\cW$ by $\wtd \cV$, $\wtd \cW$ respectively.
%It is easy to see that $\td h:\wtd\cV\to \wtd \cW$ is well-defined.  (Proposition \ref{simenv} guarantees this  if $h:\cV\to \cW$ respects direct sums. However, the assumption is not needed for the proof of this particular claim.) 
To show that $f,\td h$ satisfy the required identities one also only needs to follow the steps in the proof of Theorem \ref{IFFT}.
%Let $\s Y_1\s^{-1}=\tau Y_2\tau^{-1}$ for $Y_1,Y_2\in \cV$, $\s,\tau\in \GL_n$. 
%As $fh=\id_\cV$ we have $Y_1=f(X_1), Y_2=f(X_2)$ for $X_1,X_2\in \cW$.  
%Then 
%$$\s h(Y_1)\s^{-1}=\s h\big(f(X_1)\big)\s^{-1}=\s X_1\s^{-1}
\end{proof}

%\begin{remark}
We do not know whether $\cW$ and $\cV$ in Proposition \ref{foi} can be taken to be $G$-free sets, and consequently $h$ would be a $G$-free map; cf. \cite[Section 8]{AM2}.
%\end{remark}

\subsection{Global Free Inverse Function Theorem}
In \cite[Theorem 1.1]{Pas} it is proved that if $f$ is a $\GL$-free map and ${\rm D}f(X)$ is nonsingular for every $X\in  \cM(\CC)$ then $f$ is injective, cf. \cite{AM2}. This also holds for $\OO$-free maps. 

\begin{proposition}\label{GIFFT}
If $f:\cM(\F)^g\to \cM(\F)^{g'}$ is a differentiable $G$-free map such that ${\rm D}f(X)$ is nonsingular for every $X\in  \cM(F)$ then $f$ is injective. If $f\in \cC^r$ for some $r\in \N$ then there exists a $G$-free map $h:f(\cM(\F)^g)\to \cM(\F)^{g'}$, 
$h\in \cC^r$, such that $hf=\id|_{\cM(\F)^g}$, $fh=\id|_{f(\cM(\F)^g)}$.
\end{proposition}

\begin{proof}
Suppose that $f(X_1)=f(X_2)$ for some $X_1,X_2\in M_n(\F)^g$. 
Then \eqref{did} yields 
$$
\D f
\begin{pmatrix}
X_1&0\\
0& X_2
\end{pmatrix}
\begin{pmatrix}
0&X_1-X_2\\
X_1-X_2& 0
\end{pmatrix}
=
\begin{pmatrix}
0&0\\
0& 0
\end{pmatrix}.$$
Since $\D f
\begin{pmatrix}
X_1&0\\
0& X_2
\end{pmatrix}$ 
is nonsingular we have $X_1=X_2$, which implies the injectivity of $f$. 
The proof of the existence of the free map $h$ satisfying the required properties is the same as that of Theorem \ref{IFFT}.
\end{proof}

\begin{remark}
We remark that a free real Jacobian conjecture can be deduced from Proposition \ref{GIFFT}. 
(See e.g.~\cite[Theorem 1.3]{Pas}.)
\end{remark}
%%%%%%%%%%%%%%%%%%%%%%%%%%%%konec poskusa

\section{Examples of $\OO$-Free Maps}\label{ex}

{The theory of $\GL$-free maps is very rigid to the point that many properties are stronger
than for complex analytic functions \cite{KV,HKM11,HKM12,Voc10}. In  contrast to this is the theory of $\OO$-free maps as we shall now demonstrate. We start by presenting the following examples:
\begin{enumerate}[$\bullet$]
\item
a continuous $\OO$-free map which is not differentiable (Example \ref{cont});  more generally,
\item  $C^k$-maps which are not $C^{k+1}$ (Example \ref{k}); 
\item
a smooth $\OO$-free map which is not analytic (Example \ref{sin}).
\end{enumerate}
}

\begin{example}\label{cont}
Consider the $\OO$-free map {$f_m:\cM(\RR)\to \cM(\RR)$ defined by} 
\[f_m(x)=(xx^t)^{\frac{1}{m}} \quad\text{ for some $m\geq 2$}.\] 
It is continuous by \cite[Theorem 1.1]{ZZ}.
 Note that $f_m$ is not differentiable at $0$.
\end{example}

\begin{example}\label{k}
Let $k\in \N$ and 
$$
f:\cM(\RR)\to \cM(\RR)\quad f(x)= (xx^t)^{k+\frac{1}{2}}.
$$
Then $f$ is an $\OO$-free $C^k$-map \cite[Theorem 1.1]{ZZ},  but is not $C^{k+1}$.
\end{example}

\begin{example}\label{sin}
For an example of a smooth nonanalytic $\OO$-free map consider the map
$$
{f:\cM(\RR)\to \cM(\RR), \quad}
f(x)=\nsum_{j=0}^{\infty} e^{\displaystyle -\sqrt{2^j}}\cos\big(2^j (x+x^t)\big).$$
Since $\|\cos(2^j(A+A^t))\|\leq 1$ for every $A\in \cM(\RR)$, the power series is convergent.
We show that there exist {derivatives of all orders} in all directions at all points of $\Mr$, but $f$ is not analytic. 
Let us show {first} that $f$ is not analytic at $0$. This holds already for the function $f[1]:\RR\to \RR$.
Indeed, since
$$\limsup_{n\to \infty} \frac{|f[1]^{(n)}(0)|}{n!}\leq\limsup_{n\to \infty}\frac{e^{-\sqrt{n}}n^n}{n!}=\infty,$$
the radius of convergence of the Taylor series of $f[1]$ at $0$ is $0$.
Consider now the $\ell$-th order derivative of the function $x\mapsto \cos(kx)$ at a point $A\in M_n(\RR)$ in the direction $H\in M_n(\RR)$. 
We define matrices 
$$
A_{H}^\ell=
\begin{pmatrix}
A&H&&\\
&\ddots&\ddots&\\
&&A&H\\
&&&A\\
\end{pmatrix}\in M_{(\ell+1)n}(\RR).
$$
Let $F$ be an analytic function around $0$ with the radius of convergence $\infty$.  
The  $\ell!$-multiple of  the $(1,\ell+1)$-entry of the matrix $F(A_{H}^\ell)$ equals the $\ell$-th order derivative of $F$ at the point $A$ in the direction $H$.
By \cite[Theorem 4.25]{Hig} we have 
\[
||\cos(kA_H^\ell)||\leq (\ell+1)n\alpha k^{\ell n},
\]
where $\alpha$ depends only on $A$, 
for $A=A^t, H=H^t\in M_n(\RR)$.
This implies that 
$$\sum_{j=0}^\infty e^{-\sqrt{2^j}}\Big\|{\delta}^\ell\cos\big(2^j (A+A^t)\big)(H+H^t)\Big\|\leq 
(\ell+1)!n\alpha\sum_{j=0}^\infty e^{-\sqrt{2^j}}2^{j\ell n}< \infty.$$  
Hence the $\ell$-th order derivative of $f$ at $A$ in the direction $H$ exists and equals 
$$\sum_{j=0}^\infty e^{-\sqrt{2^j}}{\delta}^\ell\cos\big(2^j (A+A^t)\big)(H+H^t).$$
\end{example}

Let $f:\cU\to \cM(\CC)$ be an analytic $\GL$-free map. 
If $f$ is uniformly bounded on $\cU$ then the $m$-th homogeneous part of the corresponding power series is also uniformly bounded (see e.g.~the last part of the proof of \cite[Proposition 2.24]{HKM12}). 
In the case of $\OO$-free maps this is no longer the case.
\begin{example}
 The analytic $\OO$-free map
\[
x\mapsto \sin(xx^t)\]
is uniformly bounded on $\cM(\F)$, however its $(4m+2)$-th homogeneous part 
 \[(-1)^m\frac{1}{(2m+1)!} (xx^t)^{2m+1}\]
 is not uniformly bounded.
\end{example}

If an analytic $\GL$-free map $f:\cU\to \cM(\CC)$ is uniformly bounded then it converges uniformly on $\cU$ by \cite[Proposition 2.24]{HKM12}. 
The proof of the uniform convergence is easily established after noticing that the homogeneous parts of $f$ are also uniformly bounded by the same constant. 
{As the previous example shows 
this does not necessarily hold for $\OO$-free maps. Here is an explicit example of a uniformly bounded
 analytic $\OO$-free map, which does not converge uniformly in a neighborhood of $0$. }

\begin{example}
We provide an example of a bounded analytic $\OO$-free map, such that the corresponding power series converges uniformly on $M_n(\RR)$ for all $n$ but does not converge uniformly on $\cM(\RR)$. 
Define the homogeneous polynomials $z_{ij}=x_3^2x_2^{i-1}x_1^{j-1}-x_2^{i}x_1^{j}$ and let 
$$h_k(x_1,x_2,x_3)=S_{2k}(z_{11},z_{22},z_{12},z_{33},\dots,z_{kk},z_{k-1,k},z_{k+1,k+1}),$$
where $S_{2k}$
 denotes the standard polynomial of degree $2k$; i.e., 
\[
S_{2k}(x_1,\dots,x_{2n})=\sum_{\s\in {\rm Sym}(2n)}(-1)^\s x_{\s(1)}\cdots x_{\s(2n)}.
\]
We take 
\beq\label{eq:greh}
f(x_1,x_2,x_3)=\sin\Big(\sum_{k=1}^\infty k!\big(h_k(x_1,x_2,x_3)+h_k(x_1,x_2,x_3)^t\big)\Big).
\eeq
Since $S_{2k}$ is a polynomial identity of $M_n(\RR)$ for $k\geq n$ by the Amitsur-Levitzki theorem (see e.g.~\cite[Theorem 1.4.1]{Row}), $f[n]$ can be defined by taking only a finite sum in the argument of $\sin$ in \eqref{eq:greh}.  
Since $x\mapsto \sin(x)$ is analytic on $M_n(\RR)$, $f[n]$ is real analytic on $M_n(\RR)$. Moreover, $f$ is uniformly bounded by $1$, since the argument of $\sin$ in $f$ is symmetric. 
Note that the corresponding power series $F=\sum f_m$, where $f_m$  is homogeneous of degree $m$, converges uniformly on $M_n(\RR)^3$ for every $n$, since the sum in the argument of $\sin$ in the definition of $f$ is finite on $M_n(\RR)^3$ and the power series corresponding to $\sin$ restricted to symmetric matrices converges uniformly. 

We will now show that $F$ does not converge uniformly on $\cM(\RR)$.
{Assume for the sake of contradiction} that for every $\varepsilon>0$ there exist $N$ and $r>0$ such that
$$\Big\|f(X)-\sum_{m=0}^n f_m(X)\Big\|< \varepsilon \;\text{ for every $\|X\|<r$, $n\geq N$}.$$
Fix  $\varepsilon<1$ and the corresponding $N$ and $r$.  
Take $n>N$ such that 
\beq\label{eq:weirdo}
n!\left(\frac{r}{2}\right)^{2n^2+3n+1}>\frac{\pi}{2}.
\eeq
Let 
$$x_1=\sum_{i=1}^n e_{i,i+1},\quad x_2=\sum_{i=1}^{n}e_{i+1,i},\quad x_3=\sum_{i=1}^{n+1} e_{ii}+e_{n,n+1}$$ 
 be elements in $M_{n+1}(\RR)$. 
Note that $z_{ij}=e_{ij}$ for $1\leq i,j\leq n+1$, $i<j$. and $z_{ii}=e_{ii}+e_{n,n+1}$. 
For $n>2$ we thus have 
\[
h_k(x_1,x_2,x_3)=0 \quad \text{ for $k\neq n$, } \quad h_n(x_1,x_2,x_3)=(-1)^{n-1}(n+1)e_{1,n+1},
\]
where the last identity follows by the identities
\begin{multline*}
S_{2n}(e_{11},e_{22},e_{12},e_{33},\dots,e_{k-2,k-1},e_{n,n+1},e_{k-1,k},\dots,e_{n-1,n},e_{n+1,n+1})\\
=S_{2n}(e_{n,n+1},e_{22},e_{12},\dots,e_{n-1,n},e_{n-1,n},e_{n+1,n+1})=(-1)^{n-1}e_{1,n+1}
\end{multline*}
for $2\leq k\leq n+1$, and setting $e_{01}=e_{11}$. 
{By \eqref{eq:weirdo} there is}
 $r'<r$ such that 
 \[(n+1)!\left(\frac{r'}{2}\right)^{2n^2+3n+1}=\frac{\pi}{2}.
 \]
Letting
\[y_i=\frac{r'}{2}x_i,\quad 1\leq i\leq 3,
\]
we have $||y||<r$ and 
$$h_{n}(y_1,y_2,y_3)= (-1)^{n-1}\Big(\frac{r'}{2}\Big)^{2n^2+3n+1}(n+1)e_{1,n+1},$$
whence 
\[f(y_1,y_2,y_3)=(-1)^{n-1}(e_{1,n+1}+e_{n+1,1}).\]
Note that $f_m(A_1,A_2,A_3)=0$ for $m<\ell$ if $h_{k}(A_1,A_2,A_3)=0$ for $k<\ell$.
Thus, \[\sum_{m=0}^N f_m(y_1,y_2,y_3)=0\] and 
$$\Big\|f(y_1,y_2,y_3)-\sum_{m=0}^n f_m(y_1,y_2,y_3)\Big\|=1>\varepsilon,$$
 a contradiction.
\end{example}

\appendix
\section{$\U$-Free Maps}\label{apU}
In this section we give a sample of the minor modifications needed to handle the case $G=\U=(\U_n)_{n\in\N}$,  $\F=\CC$. 
 The free algebra with trace with involution over $\CC$ consists of noncommutative polynomials in the variables $x_k,x_k^*$ 
over the polynomial algebra $T^*$ in  the variables $\tr(w)$, where $w\in \langle X,X^*\rangle/_{\cyc}$, 
with the involution $\tr(w)^*:=\tr(w^*)$, $\alpha^*=\ol\alpha$  for $\alpha\in \CC$. 
The evaluation map from the free algebra with involution with trace to $M_n(\CC)$ respects involution, in particular, $\tr(A^{w^*})=\tr(A^w)^*=\ol{\tr(A^w)}$. 

It follows from \cite[Theorem 11.2]{Pro} that  a polynomial map in the variables $x_{ij}^{(k)},(x_{ij}^{(k)})^{*}$ is a $U_n$-concomitant 
if and only if it is a trace polynomial in the variables $x_k,x_k^*$, and 
nontrivial trace identities in the variables $x_k,x_k^*$ of $M_n(\CC)$ first appear in the degree $n$. 
Note that functions in commutative complex variables $x_{ij}^{(k)}$ that are real analytic can be expressed as power series in the variables $x_{ij}^{(k)},(x_{ij}^{(k)})^{*}$. With this observation and the previous statements the proofs of the following proposition and theorem go along the same lines as  the  proofs of analogous results (Proposition \ref{hompol}, Theorem \ref{rumanalit}) in the cases $G=\GL$, $G=\OO$.   

\begin{proposition}\label{Upol}
Let $f:\cM(\CC)^g\to \cM(\CC) $ be a $\U$-free map such that $f[n]$ is a polynomial map in the variables $x_{ij}^{(k)}, ({x_{ij}}^{(k)})^*$  for every $n\in \N$, 
and $\max_n \deg f[n]= d$, then $f$ is a free polynomial of degree $d$ in the variables $x_k,x_k^*$.
\end{proposition}

\begin{theorem}
Let $f:\cU\to \cM(\CC)$ be an $\RR$-analytic $\U$-free map, and let  $\cB(A,\delta)\in \cU$, $A\in M_n(\CC)^g$, $\delta=(\delta_s)_{s\in\N}$, $\delta_s>0$ for every $s\in \N$. There exist $f_m\in \CC\langle A,A^*\rangle\ast\CC\X$ 
and a formal power series 
$$F(X)=\sum_{m=0}^\infty f_m(X-A),$$
which converges in norm in a neighbourhood $\cB(A,\delta)$ of $A$ such that $F(X)=f(X)$ for $X\in \cB(A,\delta)$.  
\end{theorem}

\begin{remark}
If $f$ is a $\U$-free polynomial map (i.e., for every $n\in \N$, $f[n]$ is a polynomial map in $x_{ij}^{(k)}$, $1\leq i,j,\leq n$, $1\leq k\leq g$) of bounded degree, then  $f$ is a polynomial  in the variables $x_k,x_k^*$ by Proposition \ref{Upol}. However, as $f$ is a polynomial map, it does not involve conjugate variables, so $f$ is a  polynomial in the variables $x_k$. This also follows from the fact that $\U_n$ is Zariski dense in $\GL_n$. 
Therefore $\U$-free $\CC$-analytic maps are fairly close to $\GL$-free $\CC$-analytic maps. 
\end{remark}

\subsubsection*{Acknowledgments.}
This paper was written while the second author was visiting the University of Auckland.
She would like to thank the first author for the hospitality and the inspiring atmosphere. 
The authors thank Dmitry Kaliuzhnyi-Verbovetskyi, Jim Agler, Victor Vinnikov and Matej Bre\v sar for sharing their expertise.

\end{document}